\documentclass[a4paper,11pt]{article}

\usepackage[english]{babel}
\usepackage{amsmath,amsthm, amssymb, amsfonts}
\usepackage{setspace}
\usepackage{anysize}
\usepackage{url}
\usepackage{graphicx}
\usepackage[colorlinks=true]{hyperref}
\usepackage{mathptmx}
\usepackage{paralist}
\usepackage{verbatim}
\usepackage[utf8]{inputenc}

\theoremstyle{plain}
\newtheorem{thm}{\bf Theorem}[section]
\newtheorem{thmnonumber}{\bf Main Theorem}

\newtheorem{prop}[thm]{\bf Proposition}
\newtheorem{lemma}[thm]{\bf Lemma}
\newtheorem{corollary}[thm]{\bf Corollary}

\theoremstyle{definition}
\newtheorem{definition}{\bf Definition}[section]
\newtheorem{remark}[thm]{\bf Remark}
\newtheorem{example}[thm]{\bf Example}
\def \ad{\operatorname{depth}}
\def \Hd{\operatorname{hdepth}}
\def \cd{\operatorname{cdepth}}

\def \link{\operatorname{link}}
\def \sd{\operatorname{sd}}
\def \del{\operatorname{del}}
\def \FF{\mathbb{F}}
\def \CC{\mathbb{C}}
\def \adf{\operatorname{adepth}_{ \mathbb{F} }}
\def \depth{\operatorname{depth}}
\def \int{\operatorname{int}}
\begin{document}

\title{Discrete Morse Theory for Manifolds with Boundary}
\author{Bruno Benedetti
\\
\small Inst.\ Mathematics, MA 6-2, TU Berlin\\
\small \url{benedetti@math.tu-berlin.de} 
} 
\date{October 2, 2010}
\maketitle

\begin{abstract}
\noindent We introduce a version of discrete Morse theory specific for manifolds with boundary. The idea is to consider Morse functions for which all boundary cells are critical. We obtain ``Relative Morse Inequalities'' relating the homology of the manifold to the number of interior critical cells. We also derive a Ball Theorem, in analogy to Forman's Sphere Theorem. The main corollaries of our work~are:
\begin{compactenum}[ \rm (1) ]
\item For each $d \ge 3$ and for each $k \ge 0$, there is a PL $d$-sphere on which any discrete Morse function has more than $k$ critical $(d-1)$-cells.  

(This solves a problem by Chari.) 

\item For fixed $d$ and $k$, there are exponentially many combinatorial types of simplicial $d$-manifolds (counted with respect to the number of facets) that admit discrete Morse functions with at most $k$ critical interior $(d-1)$-cells.  

(This connects discrete Morse theory to enumerative combinatorics/discrete quantum gravity.)

\item The barycentric subdivision of any constructible $d$-ball is collapsible. 

(This ``almost'' solves a problem by Hachimori.)
\item Every constructible ball collapses onto its boundary minus a facet.

(This improves a result by the author and Ziegler.)
\item Any $3$-ball with a knotted spanning edge cannot collapse onto its boundary minus a facet. 

(This strengthens a classical result by Bing and a recent result by the author and Ziegler.)

\end{compactenum}
\end{abstract}

\enlargethispage{6mm}
{\small \tableofcontents}

\section{Introduction}
Morse theory analyzes the topology of a manifold $M$ by looking at smooth functions $f: M \longrightarrow \mathbb{R}$ whose critical points are non-degenerate. It was introduced by Marston Morse in the Twenties \cite{Morse}; it was developed in the second half of the $20^{\textrm{th}}$ century by Bott, Thom, Milnor, Smale, Witten, Goresky and MacPherson (among others); it was used by Smale and Freedman to prove the Poincar\'{e} conjecture in all dimensions higher than three.  Steven Smale called Morse theory ``the most significant single contribution to mathematics by an American mathematician'' \cite[p.~21]{GoreskyMacPherson}. 
Many functions are Morse; in fact, any smooth function can be approximated by a Morse function. Bott weakened Morse's original definition in order to allow also functions whose critical cells form submanifolds rather than isolated points (e.g. constant functions); see \cite[pp.~159--162]{Guest} and \cite{GoreskyMacPherson}. 

Any Morse function provides information on the homology of a manifold. For example, the number of critical points of index $i$ is not less than the $i$-th Betti numbers of the manifold. These bounds are usually not sharp, but in some cases they can be made sharp using Smale's cancellation theorem: See e.g.~\cite[pp.~196--199]{Guest}, \cite{RourkeSanderson} or \cite{Bott0}.  In one case Morse theory even reveals the homeomorphism type: Reeb showed that if a Morse function on a smooth manifold without boundary has only two critical points, then the manifold is a $d$-sphere. 
The Reeb theorem is a crucial ingredient in Milnor's proof of the existence of ``exotic spheres'' \cite{MilnorExotic} and in Smale's proof of the generalized Poincar\'{e} conjecture \cite{Smale}.

\enlargethispage{5mm}
The idea of duality is at the core of Morse Theory. If $f$ is a smooth Morse function on a topological manifold without boundary $M$, so is $-f$.  
The critical points of $f$ and $-f$ are the same, except that they have complementary index. This way Poincar\'{e} duality (``the $i$-th and the $(d-i)$-th Betti number of any manifold without boundary coincide'') can be seen as a corollary of Morse theory.  Moving downwards along some gradient flow line of $f$ is the same as moving upwards along some gradient flow line of $-f$. 

In the Nineties Forman introduced a discretized version of {Morse} theory, in which topological manifolds are replaced by regular CW complexes and smooth functions are replaced by special poset maps between the face poset of the complex and $(\mathbb{R}, \le)$. `Critical points' are just cells at which such maps are strictly increasing \cite{FormanWIT}. 
Moving downwards along some gradient flow line is rendered in the discrete setting by an acyclic matching of the face poset (see Figure \ref{fig:DiscreteMorseFunctions}). 
The work of Shareshian \cite{Shareshian} yields a partial discrete analogue to Bott's relative Morse theory: Given a discrete Morse function on a simplicial complex, the subcomplexes consisting only of critical cells can be ``divided out'' \cite[pp.~686--687]{Shareshian}. 

Via discrete Morse functions we can bound the homology of any regular CW complex \cite[pp.~107--108]{FormanADV}. The $i$-th Betti number of the complex never exceeds the number of critical $i$-cells. As in Reeb's theorem, in one lucky case we can even tell the homeomorphism type of the complex: Triangulated $d$-manifolds without boundary that admit a discrete Morse function with only two critical cells are, in fact, triangulated spheres. This was first proven by Forman in \cite[p.~112]{FormanADV};    an alternative proof is given in~\textbf{Theorem~\ref{thm:SphereTheorem}}.

Discrete Morse theory opens doors to computational geometry \cite{Engstroem, JerseMramorKosta, JoswigPfetsch, LewinerEtAl} 
and to commutative algebra \cite{BatziesWelker, JoellenbeckWelker, Skoeldberg}. 
At the same time, it builds on a classical notion in Combinatorial Topology, the notion of \emph{collapse} introduced by J.H.C.~Whitehead in the Thirties \cite{Whitehead}. 
A complex $C$ is \emph{collapsible} if it admits a discrete Morse function with only one critical cell. All collapsible complexes are contractible, but the converse does not hold, as shown by Zeeman's dunce hat \cite{BenedettiLutz, ZeemanDunceHat}. 
All $2$-balls are collapsible, but for each $d \ge 3$ some PL $d$-balls are not collapsible. 
However, Whitehead  proved that every collapsible PL $d$-manifold is a ball \cite{Whitehead}. PL stands here for ``piecewise-linear''; essentially, it means that the link of any face is either a ball or a sphere. Whether the PL assumption is necessary or not in Whitehead's result still represents an open problem. However, the collapsibility assumption is necessary: Newman  \cite{Newman2} and Mazur \cite{Mazur} showed that the boundary of a contractible PL manifold is a homology sphere that need not be simply connected. 

Within the PL framework,  discrete Morse functions dualize. With any triangulation of $M$ we can associate a ``dual complex'' $M^*$, a cell complex with the same underlying space as $M$.  A discrete Morse function $f$ on a PL manifold without boundary $M$ yields a discrete Morse function  $-f$  on the dual block decomposition~$M^*$; the critical cells of $f$ and $-f$ are dual to one another  \cite[pp.~111--112]{FormanADV}.

So far so good. However, there is some bad news. An extremely common operation in combinatorial topology is to ``patch'' together two $3$-balls by identifying a boundary disk of the first one with a boundary disk of the second one. The resulting complex is still a $3$-ball. More generally, for any positive integer $d$ one can ``patch'' together two $d$-balls alongside a $(d-1)$-ball in their boundary, cf.~\cite{ZeemanUCB}. Unfortunately, discrete Morse theory does not relate well to this topological operation. For example, if we patch together two collapsible $3$-balls, do we get a collapsible $3$-ball? As far as we know, this problem is open. The ``obvious'' approach, namely, to collapse the union $C = C_1 \cup C_2$ by collapsing first $C_1$ and then $C_2$, may fail. (A collapse of $C_1$ may start by removing a facet of $C_1 \cap C_2$; such a collapse is disallowed by the gluing.) In particular, we do not know whether all constructible balls are collapsible or not. 

\enlargethispage{5mm}
Another problem is the lack of control over the homotopy type of the boundary. For example: Is the boundary of a collapsible manifold always a sphere? The answer is not known. We know that the boundary of any collapsible (or contractible) manifold has the same homology of a sphere. Also, Whitehead showed that the boundary of a PL collapsible manifold is a sphere \cite{Whitehead}. However, not all collapsible manifolds are PL: For a $6$-dimensional example, take the cone over a non-PL $5$-ball.

In the present paper, we bypass these difficulties using duality. We present a version of duality for manifolds with boundary (\textbf{Theorem~\ref{thm:relativeduality}}). Let us explain the idea in the luckiest scenario, namely, when the critical cells are as few as possible. We call a manifold without boundary ``\emph{endo-collapsible}'' if it admits a discrete Morse function with two critical cells. We call a manifold with boundary  ``\emph{endo-col\-lapsible}'' if it admits a discrete Morse function whose critical cells are all the boundary cells, plus a single interior cell.  (See Section~\ref{sec:background} for details.) It turns out that the notion of endo-collapsibility is \emph{selfdual} for manifolds without boundary, and essentially \emph{dual to collapsibility} for manifolds with boundary. In fact:
\begin{compactitem}[---]
\item A PL manifold without boundary $M$ is endo-collapsible if and only if $M^*$ is \cite[pp.~111--112]{FormanADV}.
\item If $M$ is a PL endo-collapsible manifold with boundary, then $M^*$ is PL collapsible. (The converse is false, cf.~Corollary~\ref{cor:StrictlyC-EC}.) 
\item If $M$ is PL collapsible, then $M^*$ is PL endo-collapsible (Corollary~ \ref{cor:DualityEC-C}). 
\end{compactitem}

\noindent Now, the collapsibility notion can be extended to regular CW complexes, and even further~\cite[pp.~136--137]{FormanADV}. In contrast, endo-collapsibility requires the presence of a boundary, so we may define it only for pseudo-manifolds. So why bother in studying endo-collapsibility rather than collapsibility? We found four concrete reasons.

\begin{compactenum}[\rm (1)]
\item Endo-collapsibility determines the topology of a manifold. By using a result of simple-connected\-ness (Lemma~\ref{lem:SimplyConnected}) and the Poincar\'{e} conjecture, we can prove the following statement.

\par \medskip \par
\begin{thmnonumber} [Theorem \ref{thm:FirstBall}] \label{mainthm:A1}
All endo-collapsible manifolds with boundary are balls.
\end{thmnonumber} \par \medskip \par

\noindent So the dual approach pays off: Whitehead's approach can show that collapsible manifolds with boundary are balls, but only after adding the PL assumption. 

\item Endo-collapsibility behaves very well under patching. The union of endo-collapsible balls that intersect in a codimension-one endo-collapsible ball is again endo-collapsible. More generally, discrete Morse functions for which the whole boundary is made of critical cells ``naturally add up'' (see \textbf{Theorem \ref{thm:DepthUnion}}). Reflecting on the structure of the barycentric subdivision, we reach the following result.

\par \medskip \par
\begin{thmnonumber} [Theorems \ref{thm:constructibleendo} and \ref{thm:SDcollapsible}] \label{mainthm:A2}
All constructible manifolds with boundary are endo-collapsible. Their barycentric subdivisions are collapsible. 
\end{thmnonumber} \par \medskip \par

\noindent The final claim of Main Theorem~\ref{mainthm:A2} ``almost'' answers the open problem of whether all constructible $d$-balls are collapsible. (Now we know that a possible counterexample, if there is any, must be in the class of non-collapsible balls with collapsible subdivision.) Also: When paired with Main Theorem~\ref{mainthm:A1}, Main Theorem~\ref{mainthm:A2} strengthens a classical result by Zeeman (``\emph{all constructible manifolds with boundary are balls}''). 
Both Zeeman's result and the first part of Main Theorem \ref{mainthm:A2} (but not Main Theorem \ref{mainthm:A1}) generalize to pseudo-manifolds with boundary. 

\item For $3$-balls, endo-collapsibility is an intermediate property between constructibility and collapsi\-bility (cf.~Cor.~\ref{cor:StrictlyC-EC}). Now, many $3$-balls contain an interior edge (called ``knotted spanning'') that together with some boundary path forms a knot. If the knot is the sum of $2^r + 1$ trefoils, then the $r$-th barycentric subdivision of the $3$-ball is \emph{not collapsible}.  This surprising obstruction, discovered in the Sixties by Bing \cite{BING, GOO}, is sharp for all $r$ \cite{LICKMAR}. In contrast, if the knot is the sum of $2^{r-1}+1$ trefoils, then the $r$-th barycentric subdivision of the $3$-ball is \emph{not constructible} \cite{EH, HS, HZ}. (For $r=1$ this bound is ``off by one'' \cite{HZ}[Lemma~1]; we do not know whether it is sharp for some $r$.) Is there an analogous theorem for endo-collapsibility, with an intermediate bound between $2^{r-1}+1$ and $2^r + 1$? The answer is positive:

\par \medskip \par
\begin{thmnonumber} [Corollary~\ref{cor:EndoNoKnot1}] \label{mainthm:A4}
If the knot is the sum of $2^{r}$ trefoils, the $r$-th barycentric subdivision of a $3$-ball with a knotted spanning edge is \emph{not endo-collapsible}.
\end{thmnonumber} \par \medskip \par

This bound is sharp for $r=1$ (Corollary~\ref{cor:sharp}). Therefore, it traces a concrete line of distinction between constructible, endo-collapsible and collapsible  $3$-balls. The barycentric subdivision of a collapsible $3$-ball is always endo-collapsible (\textbf{Prop.~\ref{prop:subdivideC-EC}}), but not always constructible   
(\textbf{Rem.~\ref{rem:subdivideC-nonC}}).

\enlargethispage{3mm}

\item  Taking cones trivially induces collapsibility, as the cone over \emph{any} complex is collapsible. In contrast, taking cones preserves endo-collapsibility without inducing it: 

\par \medskip \par
\begin{thmnonumber} [Theorem~\ref{thm:coneendo2}] \label{mainthm:A3}
The cone over a manifold $M$ is endo-collapsible if and only if $M$ is endo-collapsible.
\end{thmnonumber} \par \medskip \par

\noindent This leads to plenty of examples of balls that are collapsible without being  endo-collapsible. If we focus on discrete Morse function for which all boundary cells are critical, studying $M$ or $v \ast M$ is the same.  (See Theorem~\ref{thm:coneendo2}.)

\end{compactenum}

\par \smallskip
\noindent Endo-collapsible $d$-balls are characterized by the fact that they collapse onto their boundary (which is a sphere) minus a facet. However, if $d \ge 4$ this complex need not be collapsible. In fact, a ``badly triangulated'' sphere can still be the boundary of a ``nicely triangulated'' ball. (See Remark~\ref{rem:Pachner}.) Are there endo-collapsible $d$-balls that are not collapsible? We do not have any examples at the moment.
We have however the following \emph{two} hierarchies:
\begin{thmnonumber} [Theorem \ref{thm:endo-hierarchy}] \label{mainthm:B}
For $d$-manifolds ($d \ge 2$),
{\em
\[
\textrm{shellable} \, \Longrightarrow \,
\textrm{constructible} \, \Longrightarrow \,
\textrm{endo-collapsible} \, \Longrightarrow \,
\left\{
	\begin{array}{l}
	\textrm{Cohen-Macaulay} \\
	\textrm{locally constructible} \, \Longrightarrow \, \textrm{simply connected}, 
	\end{array} 
\right.
\]}  
and for $d \ge 3$ all of these implications are strict.  
\end{thmnonumber} 

The local constructibility notion mentioned above was introduced in connection with discrete quantum gravity in \cite{DJ} and later linked to combinatorial topology in \cite{BZ}. For any fixed $d \ge 2$, arbitrary simplicial $d$-manifolds are much more than exponentially many, when counted with respect to the number of facets \cite[Chapter~2]{Benedetti-diss}. However, there are only exponentially many different combinatorial types of locally constructible simplicial $d$-manifolds \cite{Benedetti-diss, BZ, DJ}. Thus we essentially control the number of endo-collapsible manifolds. 

A knotted triangulation of the $3$-sphere cannot be locally constructible, if the knot is complicated enough \cite{BZ}. In Section \ref{sec:LC}, we extend this result to manifolds with arbitrary discrete Morse functions, thus answering a question of Chari \cite{Chari}. 

\begin{thmnonumber} [Theorem~\ref{thm:Obstruction}] \label{mainthm:C}
For each integer $d \ge 3$ and for each $m \ge 0$, there is a PL $d$-sphere on which \emph{any} discrete Morse function has more than $m$ critical $(d-1)$-cells. 
\end{thmnonumber}

\noindent The same can be said of $d$-balls, if we replace ``critical $(d-1)$-cells'' by ``critical interior $(d-1)$-cells'' and we restrict the conclusion to  \emph{boundary-critical} Morse functions, that is, functions for which all boundary cells are critical.

We also provide the first known exponential bound for the total number of manifolds with a given boundary-critical discrete Morse function: 

\begin{thmnonumber} [Theorem \ref{thm:enumeration}] \label{mainthm:D} 
Let $m$ and $d$ be nonnegative integers. Counting with respect to the number of facets,
there are (at most) \emph{exponentially many} combinatorial types of simplicial $d$-manifolds that admit a boundary-critical Morse function $f$ with at most $m$ interior critical $(d-1)$-cells.
\end{thmnonumber}

\enlargethispage{5mm}
\noindent An exponential lower bound can be produced easily. In fact, the $d$-spheres that are combinatorially equivalent to the boundary of some stacked $(d+1)$-polytope are already exponentially many \cite[Cor.~2.1.4]{Benedetti-diss}. The boundary of any polytope admits a discrete Morse function without critical $(d-1)$-cells, cf.~Example~\ref{ex:polytope}.

In conclusion, relative discrete Morse Theory yields a nice exponential cutoff for the class of all triangulated $d$-manifolds, much like ``bounding the genus'' yields a nice exponential cutoff for the class of triangulated surfaces \cite{ADJ} \cite[p.~7]{Benedetti-diss}. So, when the number of facets is sufficiently large, $d$-manifolds on which every discrete Morse function has a high number of critical (interior) $(d-1)$-cells are the rule rather than the exception.

\section{Background} \label{sec:background}

\subsection{Polytopal complexes} 
A \emph{polytope} $P$ is the convex hull of a finite set of points in some $\mathbb{R}^k$.
A \emph{face} of $P$ is any set of the form
$F = P \; \cap \; 
\{ \mathbf{x} \in \mathbb{R}^k :
   \mathbf{c \cdot x} = c_0 \}$,
provided $\mathbf{c \cdot x} \leq c_0$ is a linear inequality satisfied by all
points of $P$. The \emph{face poset} of a polytope is the set of all its faces, partially ordered with respect to inclusion. Two polytopes are \emph{combinatorially equivalent} if their face posets are isomorphic as posets.
A \emph{polytopal complex} is a finite, nonempty collection $C$ of polytopes (called \emph{cells}) in some Euclidean space $\mathbb{R}^k$, such that (1) if $\sigma$ is in $C$ the faces of $\sigma$ are also in $C$, (2) the intersection of any two polytopes of $C$ is a face of both. The dimension of a cell is the dimension of its
affine hull; the dimension of $C$ is the largest dimension of a cell of $C$. (Sometimes the cells are called \emph{faces}, with slight abuse of notation.) The inclusion-maximal cells are called \emph{facets}. 
Polytopal complexes where all facets have the same dimension are called \emph{pure}. The \emph{dual graph} of a pure $d$-complex $C$ is a graph whose nodes are the facets of $C$: Two nodes are connected by an arc if and only if the corresponding facets of $C$ share a $(d-1)$-cell. A pure complex is \emph{strongly connected} if its dual graph is connected.

A polytopal complex $C$ is called \emph{simplicial} (resp.~\emph{cubical}) if all of its facets are simplices (resp.~cubes). By $d$-\emph{complex} we mean ``$d$-dimensional polytopal complex''. The \emph{link} of a face $\sigma$ in a $d$-complex $C$ is the subcomplex $\link_C \, \sigma$ of the simplices that are disjoint from $\sigma$ but contained in a face that contains $\sigma$. 
The \emph{deletion} $\del_C \, \sigma$ is the subcomplex of the simplices disjoint from $\sigma$. The \emph{removal} of a face $\sigma$ from $C$ yields the subcomplex  $C - \sigma$ of the simplices disjoint from the interior of $\sigma$ (cf.~Fig.~\ref{fig:DualBlockDecomposition}).  
The $k$-skeleton of a $d$-complex $C$ ($k \le d$) is the $k$-complex formed by all faces of $C$ with dimension at most $k$. A subcomplex $D$ of $C$ is called \emph{$k$-hamiltonian} if the $k$-skeleta of $C$ and of $D$ are the same: For example, a spanning tree is a $0$-hamiltonian $1$-complex. The \emph{underlying space} $|C|$ of a complex $C$ is the union of all its faces. A simplicial complex $C$ is called \emph{triangulation} of any topological space homeomorphic to~$|C|$. 

\subsection{(Homology) manifolds}
Every triangulation $C$ of a compact connected $d$-dimensional manifold with boundary satisfies the following four combinatorial conditions:
\begin{compactenum}[(1)]
\item $C$ is pure and strongly connected;
\item every $(d-1)$-cell of $C$ lies in at most two facets of $C$;
\item the $(d-1)$-cells of $C$ that lie in exactly one facet of $C$ form a triangulation of the boundary of $|C|$;
\item the link of every interior (resp.~boundary) $k$-face has the same homology of a $(d-k-1)$-sphere (resp.~of a $(d-k-1)$-ball)  \cite[\S 35]{Munkres:AT}.
\end{compactenum} 
By a $d$-\emph{manifold with boundary} we mean a 
polytopal complex (whose underlying space is) homeomorphic to a topological compact connected $d$-manifold with non-empty boundary. By $d$-\emph{manifold without boundary} we mean a polytopal complex homeomorphic to a topological compact connected $d$-manifold with empty boundary.
When we do not want to specify whether the boundary is empty or not, we just write $d$-\emph{manifolds}. A $d$-\emph{pseudo-manifold} $P$ is a strongly connected $d$-complex in which every $(d-1)$-cell is contained in at most two $d$-cells;
the \emph{boundary} of $P$ is the subcomplex formed by all 
$(d-1)$-cells that are contained in exactly one $d$-cell of~$P$. As before, by ``pseudo-manifold with boundary'' we always mean a pseudo-manifold with non-empty boundary. Obviously every manifold is a pseudo-manifold, and every manifold with boundary is a pseudo-manifold with boundary. The boundary of a $d$-manifold is either the empty set or a disjoint union of $(d-1)$-manifolds without boundary. In contrast, the boundary of a pseudo-manifold need not be a pseudo-manifold. A cell of a pseudo-manifold is \emph{interior} if it does not belong to the boundary.

By a $d$-\emph{sphere} (or $d$-\emph{ball}) we mean a $d$-manifold homeomorphic to a sphere (or a ball). A \emph{homology pseudo-sphere} is a $d$-pseudo-manifold with the same homology of a sphere. For example, the link of an interior vertex in a $d$-manifold is a homology pseudo-sphere. It is an outstanding result in Algebraic Topology that all manifolds with the same \emph{homotopy} of a sphere are spheres. This was proven by Smale \cite{Smale} for $d \ge 5$, by Freedman \cite{Freedman} for $d=4$ and by Perelman \cite{Perelman, Perelman1} for $d=3$. In contrast, manifolds that are homology pseudo-spheres but not spheres exist in all dimensions $d\ge 3$.  None of them is simply connected: Every simply connected manifold with the same homology of a sphere is a sphere \cite[Corollary~7.8, p.~180]{GWhitehead}. In contrast, a manifold with the same homotopy of a ball need not be a ball: The boundary of a contractible manifold is a homology pseudo-sphere, not necessarily simply connected \cite{Newman2, Mazur, AncelGuilbault}. That said, if $M$ is a contractible $d$-manifold and  $\partial M$ is a $(d-1)$-sphere, then $M$ is a $d$-ball~\cite{AncelGuilbault}. 

\enlargethispage{5mm}

\subsection{PL manifolds and Poincar\'{e}--Lefschetz duality}
The barycentric subdivision $\sd C$ of a complex $C$ is the simplicial complex given by all chains of faces (with respect to inclusion) of $C$. If $\sigma$ is a face of $C$, we denote by $\hat{\sigma}$ the barycenter of $\sigma$ and by $\sigma^{\ast}$ the subcomplex of $\sd C$ given by all chains of faces of $C$ containing $\sigma$. Note that $\dim \sigma + \dim \sigma^{\ast} = \dim C$. The complex $\sigma^{\ast}$ is a cone with the point $\hat{\sigma}$ as apex and thus contractible. 

  
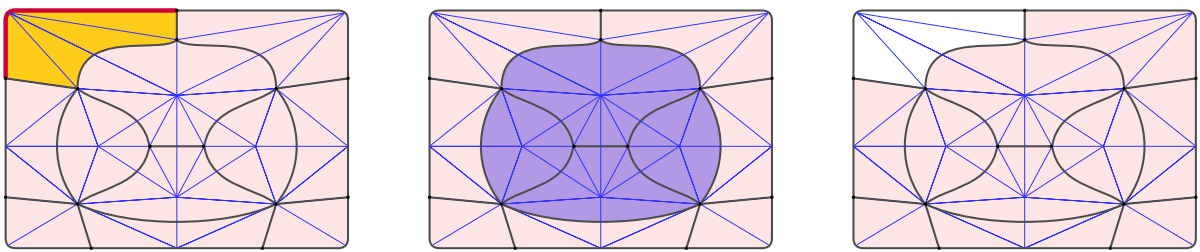
\begin{figure}[htbf]
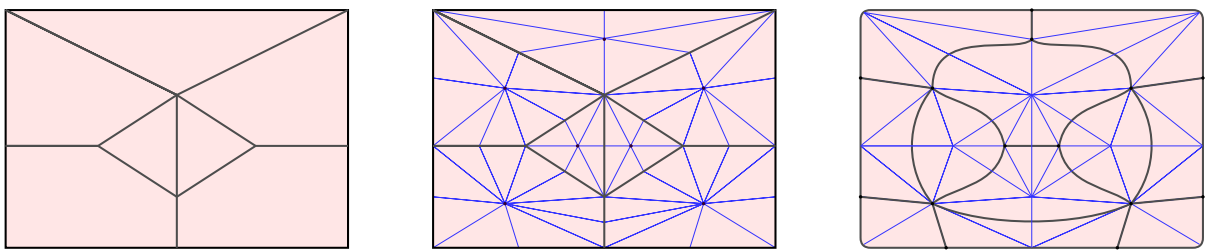

\begin{minipage}{0.99\linewidth}
\includegraphics[width=.29\linewidth]{subdivision0.eps} \hfill
  \includegraphics[width=.29\linewidth]{subdivision1.eps} \hfill
  \includegraphics[width=.29\linewidth]{subdivisionDual.eps}   

\caption{\small (\emph{1a}): A $2$-manifold with boundary $M$.  (\emph{1b}) The barycentric subdivision $\sd M$. (\emph{1c}) The dual block decomposition $M^*$.
\textit{Below}, (\emph{1d}): The cells $v^*$ (yellow) and $v^{\Diamond}$ (red), where $v$ is the top left vertex.  (\emph{1e}): The complex $M^* - (\partial M)^*$ (purple).  (\emph{1f}): Removing the dual cell of the top-left vertex $v$ yields the complex $M^* - v^*$.}

\vspace{3mm}
 \includegraphics[width=.29\linewidth]{subdivisionDualVertex.eps} \hfill
  \includegraphics[width=.29\linewidth]{subdivisionDualInt.eps}  \hfill
  \includegraphics[width=.29\linewidth]{subdivisionDualMinusV.eps}   
 \label{fig:DualBlockDecomposition} 
\end{minipage}
\end{figure}

A manifold is \emph{PL} if the links of its interior (resp.~boundary) vertices are piecewise linearly homeomorphic to the boundary of a simplex (resp.~to a simplex). In a PL manifold, the links of interior (resp.~boundary) faces are PL spheres (resp.~PL balls). 
All $d$-manifolds are PL if $d \le 4$ \cite[p.~10]{Buoncristiano}. However, some $5$-spheres and some $5$-balls are not PL \cite{EDW}. If a manifold $H$ is a homology $d$-pseudo-sphere but not a sphere, the \emph{double} suspension of $H$ is a $(d+2)$-sphere \cite{CANN} which is not PL. (The \emph{single} suspension of $H$ is a $(d+1)$-pseudo-sphere but not a manifold.) If $d \ge 5$, every PL manifold that is a homology $d$-pseudo-sphere but not a sphere bounds a \emph{contractible} PL $(d+1)$-manifold with boundary that is not a ball \cite{AncelGuilbault, Price}. Every PL $d$-sphere bounds a \emph{collapsible} PL $(d+1)$-ball \cite{Pachner0, Whitehead}; see below for the definition of collapsibility.

Let $M$ be a manifold without boundary. For any face $\sigma$ of $M$, $\sigma^{\ast}$ is a cone with the point $\hat{\sigma}$ as apex and a homology pseudo-sphere as basis. So $\sigma^{\ast}$ is a contractible pseudo-manifold, or shortly, a \emph{block}. If in addition $M$ is PL, then $\sigma^{\ast}$ is a ball. The \emph{dual block decomposition} $M^*$ is formed by the ``dual blocks'' $M^*$ of all faces of $M$ \cite[p.~377]{Munkres:AT}. It is ``almost'' a cell complex, since these blocks need not be balls (unless $M$ is PL.) $M$ and $M^*$ have the same underlying space and $(M^*)^* = M$. 

If $M$ is a $d$-manifold with boundary, $\partial M$ is a $(d-1)$-manifold without boundary. If $F$ is a $k$-cell of $\partial M$, we denote by $F^*$ the dual of~$F$ inside~$M$ and by $F^{\Diamond}$ the dual of~$F$ inside~$\partial M$. $F^{\Diamond}$ is a $(d-k-1)$-dimensional block that lies in the boundary of the $(d-k)$-dimensional block $F^*$:  See Figure \ref{fig:DualBlockDecomposition}. We denote by $(\partial M)^{\Diamond}$ the dual block decomposition of $\partial M$ with respect to itself, that is, the union of the dual blocks $F^{\Diamond}$.
The complex $M^* - (\partial M)^*$ is the union of all dual blocks~$\sigma^*$ of the interior faces $\sigma$ of $M$. The complex $M^*$ is instead the union of all dual blocks~$\sigma^*$ of the faces of $M$, plus $(\partial M)^{\Diamond}$. One has $\partial \, (M^*) = (\partial M)^{\Diamond}$: See Figure \ref{fig:DualBlockDecomposition}. Moreover, $M^* - (\partial M)^*$ is a deformation retract of $|M| - |\partial M|$ \cite[Lemma~70.1]{Munkres:AT}. Note that  $M^* - (\partial M)^*$ need not be $(d-1)$-dimensional and need not be pure.

\subsection{Regular CW complexes and (cellular) collapses}
Let $C$ be a $d$-complex. An \emph{elementary collapse} is the simultaneous removal from $C$ of a pair of faces $(\sigma, \Sigma)$, such that $\sigma$ is a proper face of $\Sigma$ and of no other face of $C$. ($\sigma$ is called \emph{free face} of $\Sigma$; some complexes have no free faces.)
We say $C$ \emph{collapses onto} $D$ if $C$ can be deformed onto $D$ by a finite sequence of elementary collapses. 
Without loss of generality, we can always assume that the pairs 
$( (d-1)\textrm{-face} , \; d\textrm{-face})$ are removed first, all pairs of the type $( (d-2)\textrm{-face} , \; (d-1)\textrm{-face})$ are removed immediately afterwards, and so on. A \emph{collapsible} $d$-complex is a $d$-complex that can be collapsed onto a single vertex. Each elementary collapse is a deformation retract, so if $C$ collapses onto $D$ the complexes $C$ and $D$ are homotopy equivalent. In particular, every collapsible complex is contractible. The converse is false: See \cite{BenedettiLutz}. However, $C$ is contractible if and only if some collapsible complex $D$ collapses also onto $C$ \cite{Whitehead}.

\enlargethispage{4mm}
Following Hatcher~\cite[p. 519]{Hatcher}, we denote by $\bar{e}^d_\alpha$ any space homeomorphic to the closed unit ball in $\mathbb{R}^d$ and by $e^d_\alpha$ the open ball obtained by removing from $\bar{e}^d_\alpha$ the boundary. A \emph{CW complex} is a space constructed via the following procedure: (1) we start with a set $X^0$ of $n$ points, the so-called \emph{$0$-cells}; (2) recursively, we form the $d$-skeleton $X^d$ by \emph{attaching} open $d$-disks $e^d_\alpha$ (called \emph{$d$-cells}) onto $X^{d-1}$, via maps $\varphi_\alpha \; : \; S^{d-1} \longrightarrow X^{d-1}$;
``attaching'' means that $X^d$ is the quotient space of the disjoint union $X^{d-1} \sqcup_\alpha \bar{e}^d_\alpha$ under the identifications $x \equiv \varphi_\alpha (x)$, for each $x$ in the boundary of $\bar{e}^d_\alpha$;
(3) we stop the inductive process at a finite stage, setting $X=X^d$ for some $d$ (called the \emph{dimension} of $X$). A CW complex is \emph{regular} if the attaching maps for the cells are injective. Every regular CW complex becomes a simplicial complex after taking the barycentric subdivision.

The notion of collapse can be extended to regular CW complexes. 
Let $X$ be a regular CW complex and $Y$ a subcomplex of $X$, such that $X$ can be represented as a result of attaching a $d$-ball $B^d$ to $Y$ along one of the hemispheres. 
The transition from $X$ to $Y$ is called \emph{elementary cellular collapse}, while the transition from $Y$ to $X$ is called \emph{elementary cellular anti-collapse}. More generally, $X$ \emph{cellularly collapses onto} a subspace $Y$ if some sequence of elementary cellular collapses transforms $X$ into $Y$. For details, see~Kozlov \cite[p.~189]{Kozlov}.

\subsection{Discrete Morse theory, collapse depth, endo-collapsibility}
A map $f: C \longrightarrow \mathbb{R}$ on a regular CW complex $C$ is a \emph{discrete Morse function on $C$} if for each face $\sigma$
\begin{compactenum}[(i)]
\item there is at most one boundary facet $\rho$ of $\sigma$ such that $f(\rho) \ge f(\sigma)$ and
\item there is at most one face $\tau$ having $\sigma$ as boundary facet such that $f(\tau) \le f(\sigma)$.
\end{compactenum}

\noindent A \emph{critical cell} of $f$ is a face of $C$ for which
\begin{compactenum}[(i)]
\item there is no boundary facet $\rho$ of $\sigma$ such that $f(\rho) \ge f(\sigma)$ and
\item there is no face $\tau$ having $\sigma$ as boundary facet such that $f(\tau) \le f(\sigma)$.
\end{compactenum}

\noindent A \emph{collapse-pair} of $f$ is a pair of faces $(\sigma, \tau)$ such that
\begin{compactenum}[(i)]
\item $\sigma$ is a boundary facet of $\tau$ and
\item $f(\sigma) \ge f(\tau)$.  
\end{compactenum} 

\vspace{-3mm}
  \begin{figure}[htbf]
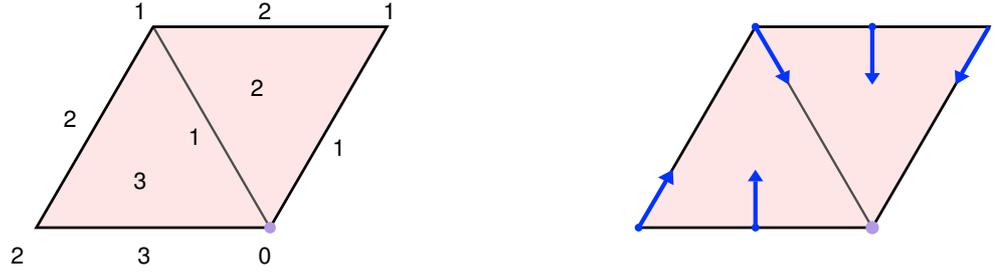

\begin{minipage}{0.99\linewidth}
\hfill \includegraphics[width=.33\linewidth]{Duality3A.eps} \hfill
  \includegraphics[width=.33\linewidth]{Twotriangles-Matching.eps}  \hfill
\caption{\small A discrete Morse function on a $2$-manifold (\emph{left}) and the collapse pairs and critical cells associated (\emph{right}). Each pair is represented by an arrow whose tail (resp.~head) is the lower-dimensional (resp.~upper-dimensional) face \cite{FormanADV}. The only critical cell in the Example above is the bottom-right vertex.}
  \label{fig:DiscreteMorseFunctions} 
\end{minipage}
\end{figure}

Forman \cite[Section~2]{FormanADV} showed that for each discrete Morse function $f$ the collapse pairs of $f$ form a partial matching of the face poset of $C$: The unmatched faces are precisely the critical cells of $f$. Each complex $K$ endowed with a discrete Morse function is homotopy equivalent to a cell complex with exactly one $k$-cell for each critical simplex of dimension $k$ \cite{FormanADV}. For example, the complex in Figure \ref{fig:DiscreteMorseFunctions} has one critical cell, which is $0$-dimensional. Therefore, it is homotopy equivalent to a point.
 
Let $M$ be a $d$-pseudo-manifold and let $f$ be a discrete Morse function on $M$. Let $k$ be an integer in $\{0, \ldots, d = \dim M\}$. By $c_k (f)$ we denote the total number of critical $k$-cells of $f$;  by $c^{\int}_k (f)$ we denote the number of critical $k$-cells of $f$ that are not contained in the boundary of $M$. When there is no ambiguity, we write $c_k$ instead of $c_k (f)$ and $c^{\int}_k$  instead of $c^{\int}_k (f)$.
We say that $f$ is \emph{boundary-critical} if all the boundary faces of $M$ are critical cells of $f$. For example, the function that assigns to each face its dimension is boundary-critical. When $\partial M = \emptyset$, any Morse function is vacuously boundary-critical.

The \emph{collapse depth} $\cd M$ of  a pseudo-manifold $M$ is the maximal integer $i$ such that the complex $M - \Delta$ collapses onto the union of $\partial M$ with a $(d-i)$-complex $C$, for some facet $\Delta$ of $M$. 
Equivalently,  $\cd M \ge k$ if and only if some boundary-critical discrete Morse function on $M$ has one critical $d$-cell and no internal critical $(d-i)$-cells, for each $i \in \{1, \, \ldots \, , \, k-1\}$.
By definition $\cd M \le \dim M$; we call a pseudo-manifold $M$ \emph{endo-collapsible} if $\cd M = \dim M$.

\subsection{Stanley--Reisner rings, algebraic depth, Cohen--Macaulayness} 
We sketch a few definitions and results, referring to Eisenbud \cite{EisenbudCA} or Miller--Sturmfels~\cite{MillerSturmfels} for details.

A \emph{Noetherian} ring is a commutative ring all of whose ideals are finitely generated. The quotient of a Noetherian ring (modulo some ideal) is still a Noetherian ring. A \emph{zero-divisor} is an element $x \ne 0$ such that $x \cdot y = 0$ for some $y \ne 0$. A \emph{length-$t$ chain} of a Noetherian ring is an increasing sequence $P_0 \subsetneq P_1 \subsetneq \ldots \subsetneq P_t$ of $t+1$ prime ideals inside the ring.  The \emph{Krull dimension} of a ring is the supremum of all the chain lenghts. A \emph{length-$t$ regular sequence} in a Noetherian ring $R$ is a list of $t$ elements $a_1, \ldots, a_t$ such that the ideal $(a_1, \ldots, a_t)$ does not coincide with the whole $R$ and in addition each $a_i$ is \emph{not} a zero-divisor in the quotient ring $R / (a_1, \ldots, a_{i-1})$. The \emph{depth} of $R$ is the supremum of all the lengths of its regular sequences. 

The integer $\dim R - \depth R$ counts also the minimal length of a finite projective resolution of $R$ (cf.~\cite[pp.~474--481]{EisenbudCA}.) In particular, $\dim R - \depth R$ is always non-negative. The rings for which $\dim R - \depth R = 0$ are called \emph{Cohen--Macaulay}.  For example, the ring $\FF[x_1, \ldots, x_n]$ is Cohen--Macaulay: A length-$n$ chain is given by $(0) \subsetneq (x_1) \subsetneq (x_1, x_2) \subsetneq \ldots \subsetneq (x_1, x_2, \ldots, x_n)$, while the variables $x_1, \ldots, x_n$ form a length-$n$ regular sequence in $\FF[x_1, \ldots, x_n]$. 

Let $\FF$ be any field. Let $C$ be a \emph{simplicial} complex with $n$ vertices, labeled from $1$ to $n$. Given a face $\sigma$ of $C$, the monomial $X_\sigma$ in the ring $\FF[x_1, \ldots, x_n]$ is the product of all variables $x_i$ such that the vertex labeled by $i$ is in $\sigma$. The \emph{Stanley--Reisner ideal} of $C$ is the monomial ideal 
$I_C 
\ := \
\{\, x_{\sigma} \; : \; \sigma \notin C \, \}$
generated by the ``non-faces'' of $C$. $I_C$ is \emph{radical}, that is, the intersection of finitely many prime ideals. Conversely, every radical monomial ideal of $\FF[x_1, \ldots, x_n]$ is the Stanley--Reisner ideal of a unique complex on $n$ vertices. The \emph{Stanley--Reisner ring}  $\FF[C]$ is the quotient of $\FF[x_1, \ldots, x_n]$ by $I_C \,$. By Hilbert's Basis Theorem \cite[p.~27]{EisenbudCA}, every Stanley-Reisner ring is Noetherian. 

The dimension of $C$ (as simplicial complex) is one less than the Krull dimension of $\FF[C]$. The \emph{algebraic depth of $C$ over $\FF$} is defined as one less than the depth of~$\FF[C]$. The simplicial complex~$C$ is called \emph{Cohen-Macaulay over $\FF$} if $\adf C= \dim C$, or equivalently, if $\depth \FF[C] = \operatorname{Krull-dim} \FF[C]$. A complex that is Cohen--Macaulay over any field is simply called \emph{Cohen-Macaulay}.
The algebraic depth of $C$ over $\FF$ equals the maximal integer $k$ for which the $k$-skeleton of~$C$ is Cohen--Macaulay \cite{Smith}. In particular, for each $d \ge 1$, a $d$-complex $C$ is connected if and only if $\adf C \ge 1$ for all fields $\FF$. In general, $\adf C$ depends on the field: For example, (any triangulation of) the projective plane is Cohen--Macaulay only over fields of characteristic $\neq 2$. The algebraic depth of $C$ does not depend on the triangulation of $C$, but only on the homology of the underlying space~$|C|$. By Hochster's formula,
\[
\adf C 
= 
\max \left\{ 
m \; : \:  \tilde{H}_i \left(\link_{C} \sigma; \FF \right) = 0 \: \textrm{ for all faces $\sigma$ and for all integers }  i \, < \, m - \dim \sigma - 1
\right\}, 
\]
where by convention $\link_{C} \emptyset = C$. So, a complex $C$ is Cohen-Macaulay over $\CC$ if and only if $\tilde{H}_i (C) = 0$ for all $i < \dim C$ and in addition, for each face $\sigma$ of $C$, one has $\tilde{H}_i \left(\link_{C} \sigma \right) = 0$ for all $i < \dim \sigma$. The latter condition is automatically satisfied is $C$ is a (homology) manifold: Hence a manifold $M$ is Cohen-Macaulay if and only if $\tilde{H}_i (M) = 0$ for all $i < \dim M$. By \cite[Corollary~7.8, p.~180]{GWhitehead} and the Poincar\'{e} conjecture, every simply connected Cohen--Macaulay $d$-manifold without boundary is a $d$-sphere.

\section{Main results} \label{sec:main}

\subsection{Relative discrete Morse inequalities}
Forman proved that every pseudo-manifold without boundary admits a \emph{polar} function, that is, a discrete Morse function with exactly one critical $d$-cell and exactly one critical $0$-cell \cite[pp.~110--111]{FormanADV}. Our first Lemma is a relative version of Forman's result. In analogy with Forman's \emph{polar Morse} functions, we propose the name \emph{equatorial Morse function} to denote a boundary-critical discrete Morse functions with $c^{\int}_0 = 0$ and $c^{\int}_d = 1$.

\begin{lemma}[Existence of boundary-critical Morse function] \label{lem:polar}
Each $d$-pseudo-manifold with boundary ($d \ge 2$) admits a 
boundary-critical discrete Morse function. Moreover, if $\Delta$ is any $d$-face of a $d$-pseudo-manifold with boundary, there is an equatorial Morse function $f$ such that $\Delta$ is the unique critical $d$-cell of $f$. 
\end{lemma}

\begin{proof}
Let $\Delta$ be a facet of a pseudo-manifold with boundary~$M$. 
Since the dual graph of $M$ is connected, we can choose a spanning tree $T$ and collapse $M - \Delta$ ``along $T$'' onto the $(d-1)$-complex $K^T$ of the $(d-1)$-faces of $M$ not crossed by $T$. Since $T$ does not cross the boundary of $M$, the complex~$\partial M$ is completely contained in $K^T$. Since $M - \Delta$ collapses onto~$K^T$, $K^T$~is connected. By~\cite[Lemma~4.3]{FormanADV}, there exists a boundary-critical Morse function $f$ on~$M$ with only one critical $d$-cell, namely, $\Delta$. However, $f$ might have one or more critical interior points. We solve this problem by modifying $f$ on a $0$-hamiltonian $1$-dimensional subcomplex of $M$.  Let $G$ be the $1$-complex obtained from $K^T$ after the removal of
\begin{compactitem}[$\,$ --- ]
\item all critical cells $\sigma$ of $f$ with $\dim \sigma \ge 2$, and 
\item all collapse pairs $(\sigma, \tau)$ of $f$ with $\dim \sigma \ge 1$.
\end{compactitem}
\noindent $G$ is a connected graph, because  $K^T$ is connected and each of the previous removals preserves connectedness. $G$ contains the whole $1$-skeleton of $\partial M$: In fact, $G$ is just the $1$-skeleton of $M$ minus all edges involved in collapse pairs of the type (edge, $2$-cell). Let $v$ be a vertex of $\partial M$. If we choose a spanning tree $T'$ of $G$ and denote by $e_1, \ldots, e_p$ the edges of $G$ that are not in $T'$, $G - e_1 - \ldots - e_p$ collapses onto $v$.  (Some of $e_1, \ldots, e_p$ may be in $\partial M$.)
So there is a Morse function $g$ on $G$ whose critical cells are only $e_1, \ldots, e_p$ and $v$. 
Next we choose a positive integer $p$ and define a map $F_p$ on the whole~$M$ setting
\[
F_p (\sigma)
\ = \
\left\{
	\begin{array}{rl}
	f(\sigma)	& \textrm{ if $\sigma$ is not in $G$,} \\
   - p + g(\sigma) & \textrm{ if $\sigma$ is in $G$.}\\
\end{array} 
\right.
\]
\noindent We leave it to the reader to check that for $p$ large $F_p$ is a boundary-critical discrete Morse function. The unique critical $0$-cell of $F_p$ is $v$, which lies on the boundary.
\end{proof}

\noindent From now on,
\begin{compactitem}
\item all the discrete Morse functions on manifolds without boundary are assumed to be polar;
\item all the boundary-critical discrete Morse functions on manifolds with boundary are assumed to be equatorial. 
\end{compactitem}

\begin{lemma}\label{lem:collar}
Let $M$ be a manifold with boundary and let $M^*$ be the dual block decomposition of $M$. Then $M$ and the complex $M^* - (\partial M)^*$ have the same homology. If in addition $M$ is PL, then $M^*$ cellularly collapses onto $M^* - (\partial M)^*$. 
\end{lemma}

\begin{proof}
Form a  list $A_0, \ldots, A_h$ of all the boundary faces of $M$, ordered by (weakly) increasing dimension. Then, remove from $M^*$ the pairs of faces $(A_i^{\Diamond}, A_i^{\ast})$, maintaining their order. If $M$ is PL, each of these steps removes a ball $A_i^{\Diamond}$ contained in the boundary of the ball $A_i^{\ast}$, of dimension one more. Therefore, each removal is an elementary cellular collapse.

If $M$ is not PL, we can still say that $M^*$ is obtained from $M^* - (\partial M)^*$ by attaching $k$-dimensional \emph{blocks}  $A_i^*$ along $\partial(A_i^*) - A_i^{\Diamond}$. Now $\partial(A_i^*)$ and $\partial A_i^{\Diamond}$ are homology pseudo-spheres (of dimension $k-1$ and $k-2$, respectively), while $A_i^*$, $A_i^{\Diamond}$ and $\partial(A_i^*) - A_i^{\Diamond}$ are homologically trivial. Hence the attachment of $A_i^*$ alongside $\partial(A_i^*) - A_i^{\Diamond}$ does not change the homology of the space. So $M^*$ and $M^* - (\partial M)^*$ have the same homology.
\end{proof}

\begin{thm} \label{thm:MorseComplex}
Let $M$ be a $d$-manifold. Let $f$ be a boundary-critical discrete Morse function on $M$. 
\begin{compactenum}[\rm (1) ] 
\item $M$ has the same homology of a cell complex with exactly $c^{\int}_d$ $0$-cells, 
$c^{\int}_{d-1}$ $1$-cells, $c^{\int}_{d-2}$ $2$-cells, $\; \ldots \;$, $c^{\int}_{1}$ $(d-1)$-cells and $c^{\int}_0$ $d$-cells.
\item If in addition $M$ is PL, then $M$ is homotopy equivalent to a cell complex with $c^{\int}_{d-k}$ $k$-cells, for each $k=0, \, \ldots \, , \, d$. 
\end{compactenum}
\end{thm}

\begin{proof}
We prove the theorem only in case $M$ is a $d$-manifold with boundary. (The case $\partial M = \emptyset$ could be shown analogously, but it was already proven by Forman \cite[pp.~103--112]{FormanADV}.) We can assume that $c^{\int}_d = 1$ and $c^{\int}_0 = 0$. Let $\Delta$ be the unique critical $d$-cell of $f$. For each $0 < k < d$, let 
$\sigma^{(k)}_{\quad 1} \,, \, \sigma^{(k)}_{\quad 2} \,, \:  \ldots  \, , \, \sigma^{(k)}_{\quad\, c^{\int}_k}$
be the list of all interior critical $k$-cells of $f$. 
From $f$ we can read off a list of removals of \emph{interior} faces:  
 \begin{compactitem}[--- ]
\item first we remove all collapse-pairs of the type $((d-1)\textrm{-face},\; d\textrm{-face} )$;
\item then we remove the $(d-1)$-faces $\sigma^{(d-1)}_{\qquad 1}, \sigma^{(d-1)}_{\qquad 2}, \ldots, \sigma^{(d-1)}_{\qquad\, c^{\int}_{d-1}}$;
\item then all collapse-pairs $((d-2)\textrm{-face }, \; (d-1)\textrm{-face} )$;
\item then the $(d-2)$-faces $\sigma^{(d-2)}_{\qquad 1}, \sigma^{(d-2)}_{\qquad 2}, \ldots, \sigma^{(d-2)}_{\qquad \, c^{\int}_{d-2}}$;
\item $\vdots$
\item then all pairs $(\textrm{edge}, \; 2\textrm{-face} )$;
\item then the edges $\sigma^{(1)}_{\quad 1}, \sigma^{(1)}_{\quad 2}, \ldots, \sigma^{(1)}_{\quad \, c^{\int}_{1}}$;
\item then all pairs $(\textrm{vertex}, \textrm{edge} )$.
\end{compactitem}

\noindent At the end of this process we are left with $\partial M$. Let us form a dual block decomposition by dualizing the previous process. We start with $X_0 : = {\Delta}^*$ and we progressively attach onto $X_0$ the dual blocks of the faces that appear in the sequence above. For example, if $(\sigma, \Sigma)$ is the first collapse-pair removed, we set $X_1 = X_0 \cup \sigma^* \cup \Sigma^*$, and so on. There are two cases to consider:

\begin{compactenum}[(I)]
\item $(\sigma, \Sigma)$ is a collapse-pair with $\dim \sigma = s$. 
If $M$ is PL, then the attachment of $\sigma^*$ (which is a $(d-s)$-ball) and $\Sigma^*$ (which is a $(d-s-1)$-ball in the boundary of $\sigma^*$)
yields a cellular anti-collapse. Cellular anti-collapses do not change the homotopy type of a space.
If $M$ is not PL, the link of $\sigma$ in $M$ might not be a sphere, thus $\sigma^*$ might not be a ball and the attachment of $\sigma^*$ might change the homotopy type. However, $\sigma^*$ is a block attached alongside the codimension-one block $\partial \sigma^* - \Sigma^*$ in its boundary. By the Mayer-Vietoris theorem, such attachment does not change the \emph{homology} of the space.

\item $\sigma$ is a critical cell. 
If $M$ is PL and  with $\dim \sigma = s$, then the attachment of $\sigma^*$ glues a $(d-s)$-dimensional ball along its whole boundary. 
If $M$ is not PL, the attachment of the dual block of a critical cell glues a block alongside its whole boundary, which is a homology pseudo-sphere. 
\end{compactenum}

\noindent Eventually, the space $Z$ obtained is the union of the dual blocks of all interior faces. In other words, $Z=M^* - (\partial M)^*$. By construction, $Z$ has the same homology of a cell complex with exactly $c^{\int}_k$ $(d-k)$-cells for each $k$. If in addition $M$ is PL, $Z$ is \emph{homotopy equivalent} to a cell complex with exactly $c^{\int}_k$ cells of dimension $(d-k)$, for each $k$. Applying Lemma \ref{lem:collar}, we conclude.
\end{proof}

\begin{corollary}[Relative (weak) Morse inequalities] \label{cor:MorseInequalities}
Let $M$ be any manifold. For any boundary-critical discrete Morse function $f$ on $M$,  
\[H_{d-k}(M) \le c^{\int}_k \, (f) \qquad \textrm{ \rm for each } k.\]
\end{corollary}

The relative (weak) Morse inequalities boil down to the classical ones if $\partial M = \emptyset$. Unless $M$ is PL, however, Theorem \ref{thm:MorseComplex} and Corollary \ref{cor:MorseInequalities} yield no information on the fundamental group of $M$. But suppose that $M$ is \emph{not} PL; suppose that some boundary-critical discrete Morse function $f$ on $M$ has no critical interior $(d-1)$-cells; is $M$ simply connected or not? We will answer this question positively in Lemma \ref{lem:SimplyConnected}. Before proving it, however, we need a combinatorial digression. 

\subsection{Asymptotic enumeration of manifolds} \label{sec:LC}
In this section we characterize manifolds with $c^{\int} (f) = m$ for some positive integer $m$ and some boundary-critical discrete Morse function $f$. (See \textbf{Theorem \ref{thm:charLCpseudo}}.) We also prove they are not so many, compared to the total number of manifolds: For $m$ fixed, we give an explicit exponential upper bound in terms of the number of facets (\textbf{Theorem \ref{thm:enumeration}}).

By a \emph{tree of $d$-polytopes} we mean a $d$-ball whose dual graph is a tree. Together with Ziegler \cite{BZ} we showed that the number of different combinatorial types of trees of  
$d$-simplices is bounded above by the $d$-th Fuss--Catalan number $ C_d (N) := \frac{1}{(d-1) N + 1}  \binom{d N}{N} $; furthermore, this exponential upper bound is essentially sharp. Therefore, for fixed $d$ there are  exponentially many trees of $d$-simplices, counting with respect to the number of facets  \cite{BZ}.  

In 1995 Durhuus and Jonsson \cite{DJ} introduced the class of \emph{LC
manifolds}, defined as follows: An LC $d$-pseudo-manifold is a pseudo-manifold obtained from a tree of $d$-polytopes by repeatedly gluing together two combinatorially equivalent adjacent $(d-1)$-faces of the boundary. ``Adjacent'' means here ``sharing at least a $(d-2)$-face'' and represents a dynamic requirement: after each identification, new pairs of boundary facets that become adjacent may be glued together. (The cell complexes consecutively formed during the gluing process might not be polytopal complexes; ignore this difficulty for the moment, or see \cite{Benedetti-diss} for details.)

The following properties follow more or less directly from the definition:
\begin{compactenum}[(1)]
\item all LC manifolds (with or without boundary) are simply connected \cite{DJ};
\item all shellable balls and spheres are LC \cite{BZ}; 
\item for fixed $d \ge 2$, there are at most exponentially many combinatorial types of LC $d$-manifolds triangulated with $N$ $\,d$-simplices  \cite{BZ}.
\end{compactenum}
\noindent This yields a non-trivial upper bound for the number of shellable manifolds. In contrast, for fixed $d \ge 2$ there are \emph{more} than exponentially many  $d$-manifolds triangulated with $N$ $\,d$-simplices \cite[pp.~45--47]{Benedetti-diss}. So, counting with respect to the number of facets, most of the $d$-manifolds are not shellable. 

In 2009, the author and Ziegler \cite{BZ} characterized LC spheres as the $3$-spheres that become collapsible after the removal of a facet. So LC $3$-spheres and endo-collapsible $3$-spheres are the same. The characterization was later extended to manifolds of arbitrary dimension \cite{Benedetti-diss}: LC $d$-manifolds coincide with the class of manifolds $M$ such that for some facet $\Delta$, $M- \Delta$ collapses onto $C \cup \partial M$, with $\dim C = d-2$. This is a weaker notion than endo-collapsibility. Here is a reformulation in the language of discrete Morse theory (where the LC notion corresponds to the case $m=0$):

\begin{thm} [cf.~{\cite[Theorem~5.5.1]{Benedetti-diss}}] \label{thm:charLCpseudo}
Let $M$ be a $d$-manifold ($d \ge 2$). Let $m$ be a non-negative integer. The following are equivalent:
\begin{compactenum}[\rm(I)] 
\item $M$ admits a boundary-critical discrete Morse function $f$, with $c^{\int}_{d-1} (f) = m$.
\item $M$ is obtained from a tree of $d$-polytopes $P$ by subsequently
	\begin{compactenum}[\rm(a)]
	\item gluing together $m$ pairs of combinatorially equivalent boundary faces of dimension $d-1$ (not necessarily adjacent);
	\item then, repeatedly gluing together two combinatorially equivalent 
			\emph{adjacent} boundary faces of dimension $d-1$.
	\end{compactenum}
\end{compactenum}
\end{thm}
\begin{proof}
Let $M$ be a $d$-manifold satisfying (II). Let $T$ be the dual graph of~$P$. Let $K^T$ be the subcomplex of $M$ given by all the $(d-1)$-faces that are not perforated by $T$. The tree of polytopes~$P$ can be recovered by ``cutting $M$ open'' along $K^T - \partial M$. We will call \emph{cut faces} the $(d-1)$-dimensional faces of $K^T$ that do not belong to $\partial M$. Each cut face $\sigma$ corresponds to two $(d-1)$-faces $\sigma', \sigma''$ of $\partial P$. (Roughly speaking, the boundary of $P$ consists of a ``single copy'' of $\partial M$ plus a ``double copy'' of $K^T - \partial M$.) We claim that:
\begin{compactenum}[\rm(i)]
\item for each facet $\Delta$ of $M$, the $d$-complex $M-\Delta$ collapses onto the $(d-1)$-complex $K^T$;
\item from phase (IIa) we can read off a removal of $m$ cut faces from $K^T$, yielding a $(d-1)$-complex $K'$ that still contains $\partial M$;
\item from phase (IIb) we can read off a sequence of elementary collapses on~$K'$ that removes all the remaining cut faces, thus yielding a $(d-1)$-complex $K''$ that contains $\partial M$ and has no $(d-1)$-face in $M - \partial M$.
\end{compactenum}
From this claim condition (I) follows easily, as $M-\Delta$ collapses onto a $(d-2)$-complex that completely contains the boundary of $M$. So all we need to show is the validity of the claim.
Item (i) follows directly from the definition of $K^T$: Just collapse $M - \Delta$ along $T$. Item (ii) can be shown step-by-step: The identification of two cut faces  $\sigma'$ and  $\sigma''$ in phase (IIa) corresponds to the removal of (the interior of) $\sigma$ from $K^T$. As for item (iii), gluing together \emph{adjacent} $(d-1)$-faces $\sigma'$ and  $\sigma''$ which share a $(d-2)$-face $F$ corresponds to sinking $F$ into the interior. We can associate to this step the elementary collapse that removes $\sigma$ together with its free face $F$ from the complex.

Conversely, suppose (I) holds: Let us show (II). By Lemma \ref{lem:polar}, we can assume that $f$ is equatorial resp. polar, according to whether $M$ is a manifold with boundary resp. without boundary. For some facet $\Delta$ of $M$ we have:
\begin{compactenum}[\rm(A)]
\item a list of elementary collapses of the type $( (d-1)\textrm{-face} , \; d\textrm{-face})$ which transforms $M-\Delta$ into some $(d-1)$-complex $K$, containing $\partial M$;
\item a list of exactly $m$ $\,(d-1)\textrm{-faces}$ in the interior of $M$ whose removals transforms $K$ onto some $(d-1)$-complex $K'$ that still contains $\partial M$;
\item a list of elementary collapses of the type $( (d-2)\textrm{-face} , \; (d-1)\textrm{-face})$ which collapses $K'$ onto some $(d-1)$-complex, all of whose $(d-1)$-faces are in $\partial M$, such that $K'$ contains $\partial M$.
\end{compactenum}

The sequence of collapses (A) acts along some spanning tree $T$ of the dual graph of $M$. Thus $K$ is the complex of the $(d-1)$-faces of $M$ not hit by $T$. This $T$ (or equivalently $K$) uniquely determines a tree of polytopes $P$ ``inside $M$''. (The dual graph of $P$ is $T$; $P$ can also be obtained cutting $M$ open alongside $K -\partial M$.) We are going to show how to obtain $M$ from $P$ via gluings of adjacent boundary facets. 

Let us label by  $b_1, b_2, \ldots, b_m$ (maintaining their order) the facets of $K$ that appear in the list (B).  Since each facet of $K$ corresponds to two facets of $\partial P$, we label by $b'_{i}$ and $b''_{i}$ the two boundary facets of $P$ corresponding to~$b_i$. We start with the tree of polytope $P$ and perform the gluings $b'_i \equiv b''_{i}$, for $i=1, \ldots, m$. A priori, there are several ways to identify two $(d-1)$-faces. However, we should glue together $b'_i \equiv b''_{i}$ exactly in the way they are identified inside $M$.

Let us label by $1, 2, \ldots, t$ (maintaining their order) the facets of $K'$ that appear in the list (C). Any of these faces (say, the one labeled by $i$) corresponds to \emph{two} facets of $\partial P$, which we label by $i'$ and $i''$. We now perform the gluings $i' \equiv i''$, for $i=1, \ldots, t$. It is not difficult to check that either $i'$ and $i''$ are adjacent, or (recursively) they have become adjacent after we glued together some $j'$ and $j''$, with $j < i$. The crucial idea is that $i'$ and $i''$ always share at least the $(d-1)$-face $F_i$ that is removed together with $i$ in the $i$-th elementary collapse of the list (C). Eventually, we re-obtain the starting manifold $M$.
\end{proof}

Based on Theorem \ref{thm:charLCpseudo}, we proceed now to bound from above the number of combinatorial types of simplicial manifolds with $c^{\int}_{d-1} (f) = m$ for some $f$. The following theorem extends results by Durhuus and Jonsson \cite[Theorem 1]{DJ} and the author and Ziegler \cite[Theorem~4.4]{BZ}. The idea, however, is analogous. We have to count certain simplicial manifolds, which are obtained from some tree of $N$ $\,d$-simplices by performing matchings in the boundary. The crucial point is, the matchings are not arbitrary: Only two particular types of matchings are allowed. Now, there are exponentially many combinatorial types of trees of $N$ $\,d$-simplices \cite{BZ}. Choose one. How many different manifolds can we obtain from it, by performing some sequence of matchings of the prescribed type?  

\begin{thm} \label{thm:enumeration}
Let $m$ and $d$ be positive integers. There are at most 
\[
(de)^N \,
2^{d(d-1)N} \, \cdot 
\left( dN - N + 2
\right)^{2m}
\, \cdot \, 
 m!
\,
\left( 
\frac{2^d e^2 d!}{2m^2}
\right)^{m}
\, 4^d
\,\] 
combinatorial types of \emph{simplicial} $d$-manifolds with $N$ facets that admit a boundary-critical Morse function $f$ with exactly $m$ critical $(d-1)$-cells in the interior. For fixed $m$ and $d$, the previous bound is singly exponential in $N$.
\end{thm}

\begin{proof}
By Theorem \ref{thm:charLCpseudo},  a $d$-manifold $M$ that admits a map $f$ as above is obtainable from a tree of $d$-simplices, by performing matchings of type (IIa) or (IIb).  As shown in \cite[Section~4]{BZ}, the number of distinct trees of $N$ $\,d$-simplices is bounded above by $(de)^N$. Set $D:=\frac{1}{2}(dN - N +2)$. A tree of $N$ $d$-simplices has exactly $2D$ boundary facets of dimension $d-1$.  Let us bound the number of matchings described as ``of type (IIa)'' in Theorem~\ref{thm:charLCpseudo}. There are 
\begin{compactitem}[---]
\item $\binom{2D}{2m}$ ways to choose $2m$ facets out of $2D$, 
\item $(2m)!! = 2m  (2m-2)  \ldots  2 = 2^m \, m!$ ways to match them, and
\item $(d!)^{m}$ ways to glue them together. In fact, a $(d-1)$-simplex has $d$ vertices, so a priori there are $d!$ ways to identify two $(d-1)$-simplices. 
\end{compactitem}
So, the total number of simplicial pseudo-manifolds with $N$ facets that can be obtained from a tree of simplices via $m$ matchings of type (IIa) is at most $(de)^N \binom{2D}{2m} \, 2^m \, m! \cdot (d!)^m.$ 
If we apply the inequality $\binom{a}{b} < a^b b^{-b} e^b$, the bound simplifies into
\begin{eqnarray} \label{eq:a1}
(de)^N \, 
\left( 
2D
\right)^{2m} \,
m! \,
\left(
\frac{e^2 \, d!}{2 m^2}
\right)^m.
\end{eqnarray}
 
We still have to bound the number of matchings described as ``type (IIb)'' in Theorem~\ref{thm:charLCpseudo}. A tree of $N$ $\,d$-simplices has exactly $dD$ pairs of adjacent boundary facets. (In fact, the boundary has exactly $dD$ faces of dimension $d-2$.) Each matching of type (IIa) creates (at most) $d$ new pairs of adjacent faces. So, before we start with type (IIb) gluings, we have a simplicial $d$-pseudo-manifold $P$ with at most $dD + dm$ pairs of adjacent $(d-1)$-cells in the boundary and with exactly $2D - 2m$ boundary facets. Proceeding as in \cite[Theorem~1]{DJ} or \cite[Section~4]{BZ}, we organize the gluings into rounds. The first round will contain couples that are adjacent in $\partial P$. Recursively, the $(i+1)$-th round will consist of all pairs of facets that become adjacent only after a pair of facets are glued together in the $i$-th round. So, the first round of identifications consists in choosing $n_1$ pairs out of $dD + dm$. After each identification at most $d-1$ new ridges are created, so that when the first round is over there are at most $(d-1)n_1$ new pairs. The second round consists in choosing $n_2$ pairs out of these $(d-1)n_1$; and so on. Eventually, we recover the starting manifold~$M$. Since we are performing a (partial) matching of the $2D-2m$ boundary facets of~$P$, $n_1 + \ldots + n_f \le D - m$. In particular, since the $n_i$'s are positive integers, there can be at most $D-m$ rounds. This way we obtain the upper bound 
\[
\sum_{f=1}^{D-m} \; \: \: \sum_{\begin{array}{c}
    n_1, \ldots, n_f \in \mathbb{Z}_{>0}\\
    \sum n_i \le D - m \\
	 n_{i+1} \leq (d-1)n_i \end{array}} \! \binom{d D + dm}{n_1}
\binom{(d-1)n_1}{n_2} \cdots \binom{(d-1)
  n_{f-1}}{n_f}
\] 
for type (IIb) matchings. Let us apply the inequality $\binom{a}{b} < 2^a$ everywhere, while ignoring the conditions $n_{i+1} \leq (d-1)n_i$. This way we reach the weaker bound 
\begin{eqnarray} \label{eq:a2}
 \sum_{f=1}^{D-m} \, \: \: \sum_{\begin{array}{c}
    n_1, \ldots, n_f \in \mathbb{Z}_{>0}\\
    \sum n_i \le D - m \end{array}} \! 2^{d D + dm} \: 2^{(d-1) (D-m)} \,.
\end{eqnarray}  
There are $\binom{D-m-1}{f-1}$ $f$-ples of positive integers whose sum is exactly $D-m$. 
Also, choosing $f$ positive integers that add up to less than $D-m$ is the same as choosing $f+1$ positive integers that add up to $D-m$. So there are $\binom{D-m-1}{f}$ $f$-ples of positive integers whose sum is less than $D-m$.

\[
\sum_{f=1}^{D-m} \: \: \sum_{\begin{array}{c}
    n_1, \ldots, n_f \in \mathbb{Z}_{>0} \\
    \sum n_i \le D-m \end{array}} 1 \quad = \quad 
\sum_{f=1}^{D-m} \; \binom{D- m-1}{f-1} +  \binom{D-m-1}{f} \: \: = \: \; 
\sum_{f=1}^{D-m} \; \binom{D-m}{f} \: \: = \: \; 
2^{D-m} -1.
\]
So, bound~(\ref{eq:a2}) boils down to $2^{Dd + dm + (d-1)(D-m) + D - m} = 2^{2Dd}$. Multiplying it with bound~(\ref{eq:a1}) we obtain the final bound
\[
(de)^N \, 
\left( 
2D
\right)^{2m} \,
m! \,
\left(
\frac{e^2 \, d!}{2 m^2}
\right)^m \,
2^{2Dd}
\]
for the total number of simplicial manifolds with $c^{\int}_{d-1} (f) = m$ for some boundary-critical Morse function $f$. The conclusion follows replacing $2D$ by $dN - N + 2$.
\end{proof}

\subsection{Sphere and Ball Theorems}
In this Section we show that endo-collapsible manifolds are either spheres or balls (\textbf{Theorem~\ref{thm:FirstBall}}). We also give a version of duality of Morse functions for manifolds with boundary (\textbf{Theorem~\ref{thm:relativeduality}}). Given a pair $(M, L)$ of the type (PL manifold, subcomplex), we show that any boundary-critical discrete Morse function on $M$ yields bounds for the homotopy of the complement of $L$ in $M$ (\textbf{Theorem~\ref{thm:RelativeObstruction}}). 

In Section \ref{sec:LC} we characterized the manifolds that admit discrete Morse functions $f$ with $c^{\int}_d (f) = 1$ and $c^{\int}_{d-1} (f) = m$. This allows us to obtain the following crucial result:

\begin{lemma}\label{lem:SimplyConnected}
Let $M$ be a $d$-manifold, $d \ge 2$. Suppose there is a boundary-critical discrete Morse function $f$ on $M$ such that $c^{\int}_d = 1$ and $c^{\int}_{d-1}=0$. Then, $M$ is simply connected.
\end{lemma}   

\begin{proof}
In view of Theorem~\ref{thm:charLCpseudo}, a manifold $M$ that admits a discrete Morse function with $c^{\int}_d = 1$ and $c^{\int}_{d-1}  =0$ is obtainable from a tree of $d$-polytopes, by repeatedly identifying two adjacent boundary facets. Now, a tree of $d$-polytopes is obviously simply connected. (It is a $d$-ball.) Moreover, if we identify boundary facets that share (at least) a $(d-2)$-face, we do not create any new loop. By induction on the number of identifications, we conclude that $M$ is simply connected.
\end{proof}

\noindent As an application, we obtain an alternative proof of Forman's Sphere Theorem. 

\begin{thm}[Forman {\cite[p.~112]{FormanADV}}] \label{thm:SphereTheorem}
Let $d \ge 2$ be an integer. Let $M$ be a $d$-manifold without boundary with a Morse function $f$ with exactly two critical cells. Then $M$ is a $d$-sphere.  
\end{thm}

\begin{proof} 
By polarity, the two critical cells have to be a $0$-cell $v$ and a $d$-cell $\Delta$ \cite{FormanADV}.  From the (relative) Morse inequalities (Cor.~\ref{cor:MorseInequalities}) we obtain that $\tilde{H}_i (M) = 0$ for all $i < \dim M$. By Lemma~\ref{lem:SimplyConnected}, $M$ is simply connected. In particular $\tilde{H}_d (M) = \mathbb{Z}$. All simply connected homology spheres are homotopy spheres and thus spheres, by the $d$-dimensional Poincar\'{e} conjecture. 
\end{proof}

\begin{remark} \label{rem:forman}
Forman's original proof \cite[p.~112]{FormanADV} uses the $d$-dimensional Poincar\'{e} conjecture, with the alternative of a direct combinatorial argument when $d=3$, using Whitehead's result that PL collapsible manifolds are balls. Here is another direct combinatorial approach: In view of Theorem~\ref{thm:charLCpseudo}, endo-collapsible manifolds are LC. In \cite{DJ} Durhuus and Jonsson showed that all (simplicial) LC $3$-manifolds without boundary are $3$-spheres. Their proof, which is elementary, can be adapted to the non-simplicial case \cite[pp.~32--34]{Benedetti-diss}. So, all endo-collapsible $3$-manifolds without boundary are $3$-spheres.
\end{remark}

Forman \cite[Theorem~4.7]{FormanADV} observed that for each discrete Morse function $f$ on a PL manifold \emph{without} boundary $M$ the function $-f$ is a discrete Morse function on the dual block decomposition $M^*$. Furthermore, $\sigma$ is a critical cell of $f$ if and only if $\sigma^*$ is a critical cell of $-f$. We may now extend these results to manifolds \emph{with} boundary:

\begin{thm}[Relative Duality] \label{thm:relativeduality}
Let $M$ be a PL manifold with boundary. Let $M^*$ be the dual of $M$. 
\begin{compactenum}[\rm(1)]
\item For each \emph{boundary-critical} discrete Morse function $f$ on $M$, there exists a (non-boundary-critical) discrete Morse function $f^*$ on $M^*$ with the following properties:
	\begin{compactitem}
	\item none of the boundary cells of $M^*$ is a critical cell of $f^*$;
	\item any interior cell $\sigma$ is a 	
	critical cell of $f$ if and only if $\sigma^*$ is a critical cell of $f^*$;
	\item $c_{d-k} (f^*) = c^{\int}_k (f)$.
	\end{compactitem}
\item For each discrete Morse function $f$ on $M$ (not necessarily boundary-critical), there exists a discrete Morse function $f^*$ on $M^*$ with the following properties:
\begin{compactitem}
	\item all of the boundary cells of $M^*$ are critical cell of $f^*$ (that is, $f^*$ is \emph{boundary-critical});
	\item any interior cell $\sigma$ is a 	
	critical cell of $f$ if and only if $\sigma^*$ is a critical cell of $f^*$;
	\item $c_{d-k} (f) = c^{\int}_k (f^*)$.
	\end{compactitem}
\end{compactenum}
\end{thm}

\begin{proof}
To prove (1) we first define a discrete Morse function $g$ on $M^* - (\partial M)^*$ by setting 
\[
g (\sigma^*) \ = \ - f (\sigma) 
\quad \textrm{for all interior cells $\sigma$ of $M$.}
\]
It is easy to see that the critical cells of $f$ and $g$ are dual to each other, hence $c_{d-k} (g) = c^{\int}_k (f)$. Now we use Lemma \ref{lem:collar} and \cite[Lemma~4.3]{FormanADV} to argue that $g$ can be extended to a function $\tilde{g}$ on $M^*$ \emph{with the same critical cells}. (See Figures~\ref{fig:DualityC-EC3} and \ref{fig:DualityC-EC4}.) Thus $c_{d-k} (\tilde{g}) = c_{d-k} (g) = c^{\int}_k (f)$. This $\tilde{g}$ is the desired $f^*$.

To show (2), note first that the cells of $M^*$ can be partitioned into two types: Those of the type $\sigma^*$, for some $\sigma$ in $M$, and those of the type $\delta^{\Diamond}$, for some $\delta$ in $\partial M$.  In order to make all cells of the type $\delta^{\Diamond}$ critical, we choose a positive integer $p$ and we define a map $F_p$ on $M^*$ as follows:  
\[
F_p (\tau)
\ = \
\left\{
	\begin{array}{rl}
	 - f(\sigma)	& \textrm{ if $\tau = \sigma^*$ for some $\sigma$ in $M$,} \\
   - p + \dim \tau & \textrm{ if $\tau = \delta^{\Diamond}$ for some $\delta$ in $\partial M$.}\\
\end{array} 
\right.
\]
Choose $p$ large enough, so that $F_p$ is a boundary-critical discrete Morse function. (See Figures~\ref{fig:DualityC-EC1} and \ref{fig:DualityC-EC2}.) Set $f^* = F_p$. A cell of the type $\sigma^*$ is critical for $f^*$ if and only if $\sigma$ in $M$ is critical for $f$. On the other hand, any cell of the type $\delta^{\Diamond}$ is critical. 
\end{proof}

\vskip-1mm
  \begin{figure}[htbf]
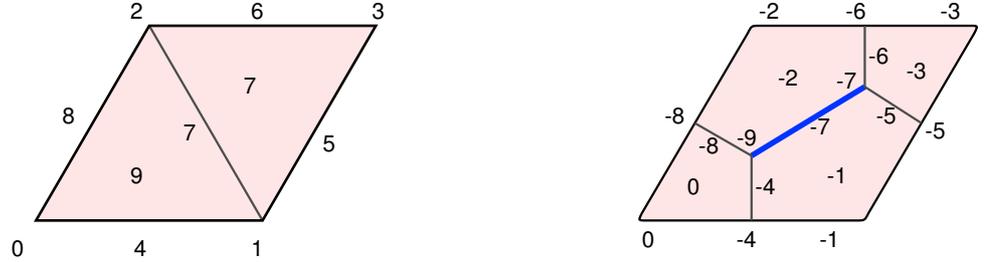

\begin{minipage}{0.99\linewidth} \hfill
\includegraphics[width=.32\linewidth]{Duality5A.eps} \hfill
  \includegraphics[width=.32\linewidth]{Duality5B.eps}  \hfill 

\caption{\small (\emph{Left}) A boundary-critical discrete Morse function $f$ on a $2$-manifold $M$ with boundary. If $\Delta$ is the left triangle, $M-\Delta$ collapses onto $\partial M$. (\emph{Right}) A corresponding discrete Morse function $f^*$ on $M^*$. The complex $M^*$ collapses onto $M^* - (\partial M)^*$, which consists of a single edge (in blue). The edge $M^* - (\partial M)^*$ collapses onto its leftmost endpoint $\Delta^*$.}
  \label{fig:DualityC-EC3} 
\end{minipage}
\end{figure}

  \begin{figure}[htbf]
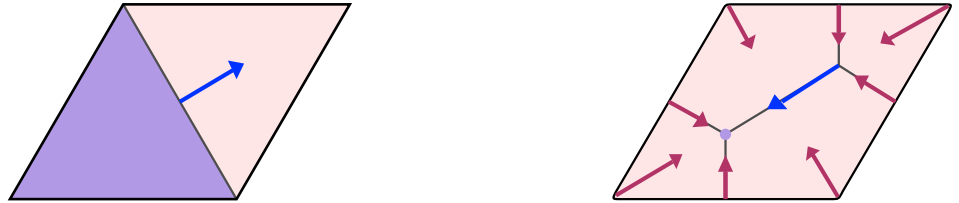

\begin{minipage}{0.99\linewidth} \hfill
\includegraphics[width=.32\linewidth]{Duality6A.eps} \hfill
  \includegraphics[width=.32\linewidth]{Duality6B.eps}  \hfill 

\caption{\small Same as Figure \ref{fig:DualityC-EC3}, but this time we draw the matchings rather than the discrete Morse functions. The blue arrow is reversed by the dualization process.}
  \label{fig:DualityC-EC4} 
\end{minipage}
\end{figure}

\begin{corollary}\label{cor:DualityEC-C}
Let $M$ be a PL $d$-manifold with boundary.
\begin{compactenum}[\rm(1)]
\item If $M - \Delta$ collapses onto the boundary $\partial M$ for some facet $\Delta$, then $M^*$ cellularly collapses onto the vertex~$\Delta^*$.
\item If $M$ collapses onto a vertex $v$, then $M^*$ minus the facet $v^*$ cellularly collapses onto $\partial M^{\Diamond}$.
\end{compactenum}
\end{corollary}

\begin{proof}
To show item (1), apply Theorem \ref{thm:relativeduality}, part (1), to the case  $c^{\int}_d (f) = 1$ and $c^{\int}_k (f) = 0$ for each $k <d$. This is illustrated in Figures~\ref{fig:DualityC-EC3} and \ref{fig:DualityC-EC4} above. (Compare also Corollary~\ref{cor:StrictlyC-EC}.)

For item (2), apply Theorem \ref{thm:relativeduality}, part (2), to the case $c_{0} (f) = 1$ and $c_i (f) = 0$ for each $i > 0$. This is illustrated in Figures~\ref{fig:DualityC-EC1}~\&~\ref{fig:DualityC-EC2} below.
\end{proof}

  \begin{figure}[htbf]
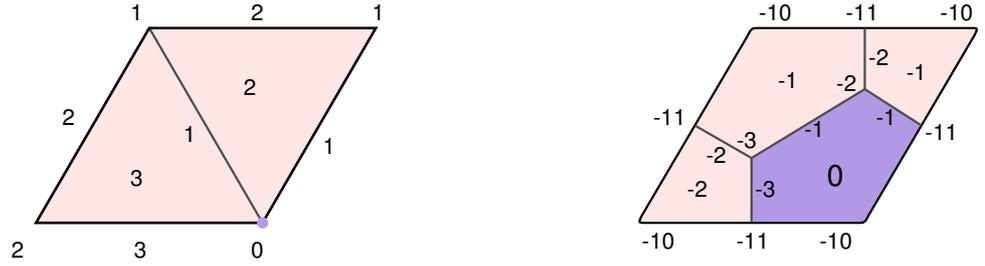

\begin{minipage}{0.99\linewidth} \hfill
\includegraphics[width=.32\linewidth]{Duality3A.eps} \hfill
  \includegraphics[width=.32\linewidth]{Duality3B.eps}  \hfill 

\caption{\small (\textit{Left}) A discrete Morse function $f$ on a $2$-manifold $M$ with boundary. $M$ collapses onto the bottom right vertex $v$ (in purple). (\textit{Right}) A corresponding boundary-critical discrete Morse function $f^*$ on $M^*$. The cell~$v^*$ is colored purple. The (pink) complex $M^*-v^*$ collapses onto the boundary of $M^*$.}
  \label{fig:DualityC-EC1} 
\end{minipage}
\end{figure}

\vspace{-1mm}
  \begin{figure}[htbf]
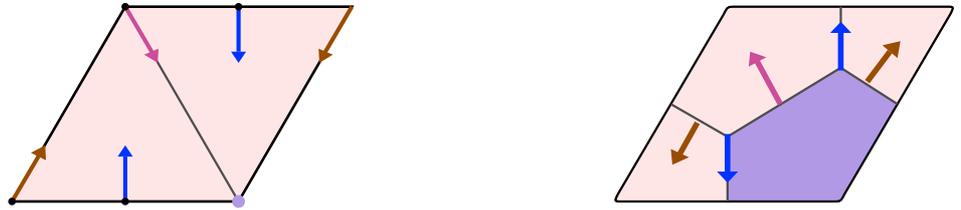

\begin{minipage}{0.99\linewidth} \hfill
\includegraphics[width=.32\linewidth]{Duality4A.eps} \hfill
  \includegraphics[width=.32\linewidth]{Duality4B.eps}  \hfill 

\caption{\small Same as Figure \ref{fig:DualityC-EC1}, but this time we draw the matchings rather than the discrete Morse functions. Dualizing reverses the arrows: Intuitively, a collapsible complex ``implodes'' to a vertex (\textit{left}), while an endo-collapsible complex ``explodes'' to its boundary after a facet is removed (\textit{right}).}
  \label{fig:DualityC-EC2} 
\end{minipage}
\end{figure}

\begin{thm}[Ball Theorem] \label{thm:FirstBall}
Let $M$ be a $d$-manifold with boundary (PL or non-PL). Let $f$ be a boundary-critical function on $M$. Suppose $f$ has only one critical interior cell. 
Then:
\begin{compactenum}[\rm(1)]
\item $M$ is a ball; 
\item if in addition $M$ is PL, then the dual block decomposition $M^*$ is cellularly collapsible.
\end{compactenum}
\end{thm} 

\begin{proof}
Let $\Delta$ be the unique ($d$-dimensional) critical cell of $f$ in the interior of $M$. 
(Up to replacing $M$ with the second barycentric subdivision $\sd^2(M)$, we can assume that $\partial \Delta$ and $\partial M$ are disjoint.) Since $f$ is boundary-critical, $M - \Delta$ collapses onto $\partial M$.  By Lemma~\ref{lem:SimplyConnected}, $M$ is simply connected. Using Seifert--Van Kampen and Mayer--Vietoris, we obtain that $M - \Delta$ is simply connected and has trivial homologies up to the $(d-1)$-st one. Since $M - \Delta$ collapses onto $\partial M$, the $(d-1)$-manifold $\partial M$ is homotopy equivalent to $M - \Delta$. Thus  $\partial M$ is a simply-connected homology $(d-1)$-sphere. By \cite[Corollary~7.8, p.~180]{GWhitehead} and the $(d-1)$-dimensional Poincar\'{e} conjecture,
$\partial M$ is a sphere. 
So $M$ is a simply connected homology ball and $\partial M$ is a sphere.
Applying \cite[Corollary~7.8, p.~180]{GWhitehead} and the $d$-dimensional Poincar\'{e} conjecture, we obtain that $M$ is a ball.
 \end{proof}

\begin{remark} \label{rem:forman1}
In the PL case (and in particular if $d \le 4$), Theorem \ref{thm:FirstBall} admits a direct combinatorial proof. In fact, by Corollary \ref{cor:DualityEC-C}, $M^*$ cellularly collapses onto the point $\Delta^*$, whence we conclude using Whitehead's theorem \cite{Whitehead}. Compare Remark~\ref{rem:forman}.
\end{remark}

Conversely to the approach of Theorem \ref{thm:FirstBall} and Theorem~\ref{thm:SphereTheorem}, suppose we are given a triangulated sphere (or a triangulated ball). Can we find ``perfect''  (boundary-critical) discrete Morse functions on them, possibly after refining the triangulation? The following is a (partial) answer.

\begin{definition}
A \emph{PL$^+$} $d$-ball is a ball piecewise-linearly homeomorphic to a $d$-simplex.
A \emph{PL$^+$} $d$-sphere is a sphere piecewise-linearly homeomorphic to the boundary of a $(d+1)$-simplex.
\end{definition}
\noindent (All PL$^+$ $d$-spheres and $d$-balls are PL. The converse is true for all $d \ne 4$, and unknown for $d=4$ \cite{Milnor}: A priori some $4$-spheres or $4$-balls might not be $PL^+$, whereas all $4$-manifolds are PL.)

\begin{thm} [{Forman \cite[Theorem~5.2]{FormanADV}}]\label{thm:SecondBall}
Let $M$ be any PL$^+$ $d$-ball or $d$-sphere. There exists an integer $k$ such that the $k$-th barycentric subdivision of $M$ is endo-collapsible.
\end{thm}

\begin{proof}
Every PL$^+$ ball or sphere admits a shellable barycentric subdivision \cite[p.~200]{BruggesserMani}. Shellable manifolds are endo-collapsible, cf.~Theorem~\ref{thm:constructibleendo} below or Chari~\cite{Chari}.
\end{proof}

\noindent Theorem \ref{thm:SecondBall} can be extended to PL balls and spheres, with a special argument for $d=4$: See~\cite{Benedetti-DMTapaMT}. 

There is however no integer $k$ such that every $3$-sphere has endo-collapsible $k$-th derived subdivision. For each $k$, one can produce a $3$-sphere whose $k$-th subdivision is not endo-collapsible and not even LC. For details, see Section  \ref{sec:knots}. The reason for this is that any discrete Morse function on a manifold $M$ gives upper bounds for the homology not only of $M$, but also of the complement of any subcomplex of $M$. 

\begin{thm} \label{thm:RelativeObstruction}
Let $d \ge 2$ be an integer. Let $M$ be a PL $d$-manifold. Let $L$ be a subcomplex of $M$, with $\dim L \le d - 2$. Let $L^{\int}$ be the subcomplex of $M$ determined by the facets of $L$ not in $\partial M$. Let $f_{k}(L^{\int})$ be the number of $k$-faces of $L^{\int}$. If $g$ is any boundary-critical discrete Morse function on $M$, let $c^{\int M \cap L}_{k} (g)$ 
be the number of critical $k$-cells of $L^{\int}$. 

The space $|M| - |L^{\int}|$ is homotopy equivalent to a cell complex with $1$ point and at most \[c^{\int}_{d-k} (g) -  c^{\int M \cap L}_{d-k} (g) + f_{d-k-1}(L^{\int})  \] 
cells of dimension $k$, for each $k=1, \ldots, d$.
\end{thm}

\begin{proof}
The argument is very similar to the proof of Theorem \ref{thm:MorseComplex} and of \cite[Theorem~2.19]{BZ}. For completeness, we sketch it anyway:
\begin{compactenum}[(i)]
\item We write down the sequence of collapses and removals of critical cells given by $g$. 
\item We ignore the faces in $L^{\int}$, while we progressively attach to one another the dual blocks of all the other faces of the sequence, maintaining their order. 
\item Eventually, we obtain a cell complex which is homotopy equivalent not the whole $M$, but just to  $|M| - |L^{\int}|$.
\end{compactenum}
\noindent For example, the list described in item (i) will start with some critical $d$-cell $\Delta$, followed by a $(d-1)$-cell $\sigma$ and another $d$-cell $\Sigma$. Since $L$ is lower-dimensional, none of $\Delta, \sigma, \Sigma$ lies in $L$. So we start with the point $\Delta^*$ and we attach onto it the $0$-cell $\Sigma^*$ and the 1-cell $\sigma^*$; and so on.

Just like in the proof of Theorem \ref{thm:MorseComplex}, the progressive attachments of the  dual blocks often preserve the homotopy type. There are in fact only two exceptional cases, in which the homotopy type changes by the addition of a cell:
\begin{compactenum}[(A)]
\item The removal of a critical $i$-face $\sigma$ that is not in $L^{\int}$ dually corresponds to the attachment of a $(d-i)$-cell $\sigma^*$ along its whole boundary.
\item An elementary collapse $(\sigma, \Sigma)$ that removes one $(i-1)$-face $\delta$ in $L^{\int}$ and one $i$-face $\Delta$ \emph{not} in $L^{\int}$ dually corresponds to the attachment of the sole $(d-i)$-cell $\Delta^*$ along its whole boundary.
\end{compactenum}   
\noindent (The case when we collapse two faces $\delta, \Delta$ both in $L$ does not change the homotopy type: dually, nothing happens. Also, if $\sigma$ is a critical cell that belongs to $L^{\int}$, we do not attach its dual, so the homotopy type is trivially preserved.) 

Case (A) occurs exactly $c^{\int}_{i} (g) - c^{\int M \cap L}_{i} (g)$ times, which is the number of critical interior $i$-cells of $g$ that are not in $L^{\int}$. Case (B) occurs at most $f_{i-1}(L^{\int})$ times, which is the number of $(i-1)$-faces $\delta$ in $L^{\int}$.  (Each of these $\delta$ is removed together with some $i$-face, which might or might not be in $L^{\int}$). Setting $k=d-i$ we conclude.
\end{proof}

\begin{example} \label{ex:ex1}
Let $m$ be an integer greater than $2$. Let $L$ be a knot realized as $3$-edge subcomplex of a $3$-sphere~$S$ (cf.~Section \ref{sec:knots}). Since $S$ has no boundary, $L^{\int} = L$. Let $g$ be \emph{any} discrete Morse function on~$S$. By Theorem  \ref{thm:RelativeObstruction}, the space $|S| - |L|$ is homotopy equivalent to a complex with one point and at most $c_2(g) + 3$ cells of dimension one.   
Thus the fundamental group of  $|S| - |L|$ has a presentation with at most $c_2(g) + 3\;$ generators. Now, the fundamental group of  $|S| - |L|$ is the so-called \emph{knot group} of $L$, a well-known invariant which depends only on the knot type and not on the triangulation chosen. Suppose the knot $L$ is so tangled that any presentation of its group requires at least $t$ generators. Then obviously $c_2(g) + 3 \ge t$. This yields the numerical condition $c_2(g) \ge t - 3$ on \emph{every} discrete Morse function $g$ on $S$.  Compare with Theorem~\ref{thm:Obstruction}. 
\end{example}

\begin{example} \label{ex:ex2}
Let $L$ be a knot realized as $m$-edge subcomplex of a $3$-ball $B$, so that only one edge $[x,y]$ of $L$ belongs to the interior of $B$. (See Section~\ref{sec:knots}.) By Theorem  \ref{thm:RelativeObstruction}, the fundamental group of $|B|-|[x,y]|$ admits a presentation with at most $c^{\int}_2(g) + 1$ generators, where $g$ is any boundary-critical discrete Morse function on $B$.  Suppose any presentation of the fundamental group of $|B|-|[x,y]|$ requires at least $t$ generators: Then for every discrete Morse function $g$ on $B$ we have $c^{\int}_2(g) + 1 \ge t$, whence $c^{\int}_2(g) \ge t-1$. Notice the gap with Example~\ref{ex:ex1}.  
\end{example}

The crucial idea for the proof of Theorem \ref{thm:RelativeObstruction} comes from a 1991 paper by Lickorish \cite{LICK}, who treated the special case when $M$ is a $3$-sphere, $L^{\int} = L$ is a $1$-sphere, $c_1(g) = c_2 (g) = 0$ and $f_{0}(L) = f_1(L)=3$. A generalization by the author and Ziegler to the case $\partial M = \emptyset$, $L^{\int} = L$ and $c_{d-1}(g) = 0$ can be found in \cite{BZ}; see also {\cite[Thm.~5.5.4]{Benedetti-diss}}. 

Theorem~\ref{thm:RelativeObstruction} represents a step towards a relative version of discrete Morse theory.  It might seem bizarre to focus on pairs $(M, L)$, with $M$ a PL manifold and $L$ a subcomplex of $M$ (not necessarily a submanifold). However, the PL assumption is necessary to activate duality, which yields the crucial tool to find models for the complement of the subcomplex. Intuitively, the fewer critical cells a Morse function on $M$ has, the simpler our model for $|M|-|L|$ will be. 
On the other hand, if we expect a priori a complicated model for $|M|-|L|$ (like in Example \ref{ex:ex1}), then any discrete Morse function on our manifold must have plenty of critical cells.

\subsection{Patching boundary-critical discrete Morse functions together}
The following theorem can be seen as an extension of a result of the author and Ziegler \cite[Lemma 2.23]{BZ}, who studied the case $k=d-1$ and $c_{d-1}(f) = c_{d-1}(g) = 0$.

\begin{thm} \label{thm:DepthUnion}
Let $M =M_1 \cup M_2$ be three $d$-pseudo-manifolds such that $M_1 \cap M_2$ is a $(d-1)$-pseudo-manifold. Let $f$ and $g$ be equatorial boundary-critical discrete Morse functions on $M_1$ and $M_2$, respectively. 
Let $h$ be a boundary-critical discrete Morse functions on $M_1 \cap M_2$, with $c^{\int}_{d-1} (h) =1$. There exists a boundary-critical discrete Morse function $u$ on $M$ with  
\[ 
c^{\int}_k (u) 
\ = \
\left\{
	\begin{array}{rl}
	c^{\int}_k (f) \; + \; c^{\int}_k (g) \; + \; c^{\int}_k (h) \;, 
	& \textrm{ if $k$ is in } \{0, \ldots, d-2\}, \\
	c^{\int}_k (f) \; + \; c^{\int}_k (g) \; + \; c^{\int}_k (h) - 1\;,
	& \textrm{ if $k$ is in } \{d-1, d\}.\\
\end{array} 
\right.
\]
\end{thm}

\begin{proof}
Choose a $(d-1)$-face $\sigma$ in $M_1 \cap M_2$. Let $\Sigma_i$ be the unique face of $M_i$ containing $\sigma$ ($i=1, 2$). Without loss of generality, we can assume that $\Sigma_1$ is the unique critical $d$-cell of $f$, $\Sigma_2$ is the unique critical $d$-cell of $g$ and $\sigma$ is the unique critical $(d-1)$-cell of $h$. 
The idea is to ``drill'' first inside $M_1 - \Sigma_1$ using the discrete Morse function $f$, then to collapse away the pair $(\sigma, \Sigma_2)$, then to drill inside $M_2 - \Sigma_2$ using $g$, and finally to drill inside $M_1 \cap M_2$ using $h$. 
Formally, we choose a positive integer $p$ and we define a real-valued map $u_p$ on $M$ as follows:
\[
u_p (\tau)
\ = \
\left\{
	\begin{array}{rl}
	g(\Sigma_2)	& \textrm{ if $\tau = \sigma$,} \\
   f(\tau) + p & \textrm{ if $\tau = \Sigma_1$,} \\
	f(\tau) & \textrm{ if $\tau$ is in the interior of $M_1$ and $\tau \neq \Sigma_1$, }\\ 
	g(\tau) & \textrm{ if $\tau$ is in the interior of $M_2$,}\\
	h(\tau) - p & \textrm{ if $\tau$ is in the interior of $M_1 \cap M_2$ and $\tau \neq \sigma$,}\\
   \dim \tau - p^2 & \textrm{ if $\tau$ is in $\partial M$.}\\
\end{array} 
\right.
\]
\noindent The previous seven cases form a partition of $M$; the seventh may be void. 

For $p$ large, the map $u_p$ achieves its maximum on $\Sigma_1$ and its minimum on all vertices of $\partial M$; moreover, for $p$ large the whole boundary of $M$ is critical. 
Since $u_p(\sigma) = g(\Sigma_2) = u_p(\Sigma_2)$, the cells $\sigma$ and $\Sigma_2$ form a collapse pair. With the exception of $\Sigma_1$ and $\Sigma_2$, $u_p$ coincides with $g$ on the interior of $M_2$ and with $f$ on the interior of $M_1$. On the interior cells of $M_1 \cap M_2$ except $\sigma$, the map $u_p$ is just $h$ shifted by the constant $-p$: This way, if $p$ is large, for each $(d-1)$-cell $\delta$ of $M_1 \cap M_2$  one has $h(\delta) - p < \min \left\{ f(\Delta_1), g(\Delta_2) \right \}$, where $\Delta_i$ is the $d$-cell of $M_i$ containing $\delta$. Also, for $p$ large, $u_p$ attains a larger value on $\sigma$ than on any of the $(d-2)$-faces in the boundary of $\sigma$. In conclusion, for $p$ sufficiently large the map $u_p$ is a boundary-critical discrete Morse function.

By construction, the interior critical cells of $u_p$ are either interior critical cells of $f$ or interior  critical cells of $g$ or interior  critical cells of $h$.  The only exceptions are the cells  $\Sigma_2$ (of dimension $d$) and $\sigma$ (of dimension $d-1$), which are critical under $g$ and $h$, respectively, but not critical under $u$.
\end{proof}

 \begin{figure}[htbf]
\begin{center}
\includegraphics[width=.37\linewidth]{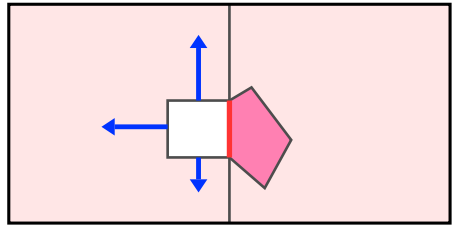} 
\caption{\small How to ``patch'' together boundary-critical discrete Morse functions. First drill on the left, then perforate the barrier, then drill on the right, and finally take away what is left of the barrier.}
  \label{fig:ConstructibleEndo}
\end{center}
\end{figure}

Next we deal with cones over manifolds. For the next result, we recall that we assume polarity (resp.~equatoriality) for all functions on manifolds without boundary (resp.~with boundary). 

\begin{thm} \label{thm:cone}
Let $f$ be a boundary-critical discrete Morse function on a manifold $M$. There is a boundary-critical discrete Morse function ``$v \ast f$'' on $v \ast M$ such that $c^{\int}_{0} (v \ast f) = c^{\int}_{1} (v \ast f) = 0$ and 
\[c^{\int}_k (f) \ = \ 
c^{\int}_{k+1} (v \ast f)
\qquad \textrm{ for each } k \textrm{ in } \{1, \ldots, d\}.
\] 
Conversely, if $g$ is a boundary-critical discrete Morse function on the pseudo-manifold $v \ast M$, there is a boundary-critical discrete Morse function $f$ on $M$, with $c^{\int}_{0} (f) = 1$ or $0$ according to whether $\partial M$ is empty or not,  and
\[c^{\int}_k (f) \ = \ 
c^{\int}_{k+1} (g)
\quad \; \qquad \textrm{ for each } k \textrm{ in } \{1, \ldots, d\}.
\] \end{thm}

\begin{proof}
Observe first that all internal faces of $v \ast M$ are of the form $v \ast \sigma$, for some $\sigma$ in $M$. Let $f$ be a boundary-critical discrete Morse function on $M$. If $M$ is a manifold with boundary and $p$ is a positive integer, define a map $u_p : M \longrightarrow \mathbb{R}$ by setting 
\[
u_p (\tau)
\ = \
\left\{
	\begin{array}{rl}
	f(\sigma)	& \textrm{ if $\tau = v \ast \sigma$ for some nonempty face $\sigma$ of $M$,} \\
   -p + \dim \tau & \textrm{ if $\tau$ is a face of $M$ or if $\tau = v$.}\\
\end{array} 
\right.
\]
If $M$ is a manifold without boundary and $x$ is the unique critical $0$-cell of $f$, set instead 
\[
u_p (\tau)
\ = \
\left\{
	\begin{array}{rl}
	f(\sigma)	& \textrm{ if $\tau = v \ast \sigma$ for some nonempty face $\sigma$ of $M$, $\sigma \neq x$} \\
   -p + \dim \tau & \textrm{ if $\tau$ is a face of $M$}\\
   -p+1 & \textrm{ if $\tau = v$ or if $\tau = v \ast x$.}\\
\end{array} 
\right.
\]
For $p$ sufficiently large, $u_p$ yields the desired boundary-critical Morse function on $v \ast M$. Notice that when $\partial M = \emptyset$ the vertex $v$ forms a collapse pair with $v \ast x$. For each $i > 0$, there is an obvious bijection between the critical $i$-cells of $M$ with respect to $f$ and critical interior $(i+1)$-cells of $v \ast M$ with respect to $u_p$. In case $M$ is a manifold with boundary, $M$ has no critical interior $0$-cell. In case $M$ is a manifold without boundary, there is a critical interior $0$-cell $v$; however, the corresponding $1$-cell $v \ast x$ of $v \ast M$ is paired with $v$. 

Conversely, given a boundary-critical discrete Morse function $g$ on $v \ast M$, 
setting $f(\sigma) := g(v \ast \sigma)$ for each $\sigma$ in $M$ one obtains a boundary-critical discrete Morse function on $M$.
\end{proof}

The barycentric subdivision of a collapsible (respectively, endo-collapsible) manifold is collapsible (resp.~endo-collapsible). More generally, a discrete (resp.~a boundary-critical discrete) Morse function on $M$ induces discrete (resp.~boundary-critical discrete) Morse functions on $\sd M$ with the same number of critical cells \cite{FormanADV}. Sometimes a triangulation ``improves'' by taking subdivisions, so $\sd M$ might admit Morse functions with \emph{fewer} critical cells: See \cite{Benedetti-DMTapaMT}.

\enlargethispage{4mm}
The last result of this section shows how to get from a \emph{boundary-critical discrete} Morse function on a manifold $M$ with boundary to a \emph{discrete} Morse function on $\sd M$. Then we show how to obtain a \emph{boundary-critical discrete} Morse function on $\sd M$ starting with a \emph{discrete} Morse function on $M$. We will need this type of results in Sections~\ref{sec:constructible} and~\ref{sec:knots}. 

\begin{thm} \label{thm:subdivision}
Let $M$ be a PL $d$-manifold with boundary. Let $F_1, \ldots, F_B$ be a list of all the non-empty boundary faces of  $M$. Let  $F_{B+1}, \ldots, F_M$ be a list all the non-empty interior faces of $M$. For each $i$ in $\{1, \ldots, B\}$, let $h_i$ be a discrete Morse function on the ball $\sd \link_M F_i$. For each $j$ in $\{1, \ldots, M\}$, let $g_j$ be a boundary-critical discrete Morse function on the manifold $\sd \link_M F_j$.
\begin{compactenum}[\rm (1)]
\item If $f$ is a \emph{boundary-critical discrete} Morse function on $M$, there exists a \emph{discrete} Morse function $F$ on $\sd M$ such that
\[
c_k (F)
\; = \;
\left\{
	\begin{array}{ll}
	c^{\int}_{d-k} (f) + 
   \sum_{i=1}^{B} c_{k-1} (h_i) + 
   \sum_{j=B+1}^{M} c^{\int}_{k-1} (g_j) 	
			& \textrm{ if $k \ge 2$} \\
   c^{\int}_{d-k} (f) 
		   & \textrm{ if $k \le 1$.}\\
\end{array} 
\right.
\]

\item If $f$ is a \emph{discrete} Morse function on $M$, there exists a \emph{boundary-critical discrete} Morse function $F$ on $\sd M$ such that
\[
c_k^{\int} (F)
\; = \;
\left\{
	\begin{array}{ll}
	c_{d-k} (f) + 
   \sum_{j=1}^{M} c^{\int}_{k-1} (g_j) 	
			& \textrm{ if $k \ge 2$} \\
   c_{d-k} (f) 
		   & \textrm{ if $k \le 1$.}\\
\end{array} 
\right.
\]

\end{compactenum}
\end{thm}

\begin{corollary} \label{cor:subdivision1}
Let $B$ be a PL $d$-ball. Suppose that $\sd \link_{B} \sigma$ is endo-collapsible, for each non-empty interior face $\sigma$ of $B$. 
\begin{compactenum}[\rm (1)]
\item If $B$ is endo-collapsible, and in addition $\sd \link_{B} \delta$ is collapsible for each non-empty face $\delta$ of $\partial B$, then $\sd B$ is collapsible.
\item If $B$ is collapsible, and in addition $\sd \link_{B} \delta$ is endo-collapsible for each non-empty face $\delta$ of $\partial B$, then~$\sd B$ is endo-collapsible.
\end{compactenum}
\end{corollary}

Here we prove only Corollary~\ref{cor:subdivision1}, which is what we really need for the applications. The proof for Theorem~\ref{thm:subdivision} is a straightforward generalization, using Theorem~\ref{thm:cone} and Theorem~\ref{thm:DepthUnion}.

\begin{proof} Set  $C:=B^* - (\partial B)^*$.
\begin{compactenum}[\rm (1)]
\item  The dual complex $B^*$ cellularly collapses onto~$C$, which in turns cellularly collapses onto a point, by Corollary \ref{cor:DualityEC-C}. All we need to show is that (i) $\sd C$ is collapsible, and (ii) $\sd B$ collapses onto $\sd C$.

To prove (i), suppose a cellular collapse of $C$ removes two cells $\Sigma^*$ and $\sigma^*$, where $\sigma$ is a codimension-one face of $\Sigma$ in the interior of $B$. Inside $\sd(B)$, $|\sigma^*|$ is triangulated as a cone with the barycenter $\hat{\sigma}$ of $\sigma$ as apex and $\sd \link_{B} \sigma$ as basis. By the assumption, $\sd \link_{B} \sigma$ and $\sd \link_{B} \Sigma$ are endo-collapsible spheres. 
By Theorem~\ref{thm:cone}, $D:=\hat{\sigma} \; \ast \;  \sd \link_{B} \sigma$ and $E:=\hat{\Sigma} \; \ast \; (\sd \link_{B} \Sigma)$ are endo-collapsible balls. So, $\partial D$ is an endo-collapsible sphere containing the endo-collapsible ball $E$. In particular, for any facet $\tau$ of $E$, $\partial D - \tau$ is collapsible.  By coning with apex $\hat{\sigma}$, we obtain a collapse of $D - \hat{\sigma} \ast \tau  - \tau$ onto $\partial D - \tau$. By the endo-collapsibility of $E$, the ball $\partial D - \tau$ collapses onto $\partial D - E$. Therefore, for any facet $\tau$ of $E$ we have a (simplicial) collapse of $D - \hat{\sigma} \ast \tau  - \tau$ onto $\partial D - E$. This means we can represent the cellular collapse $(\Sigma^*, \sigma^*)$ as a sequence of (several) elementary collapses in the barycentric subdivision of $B$. Thus the cellular collapse of $C$ onto a point $v$ translates into a (simplicial) collapse of $\sd C$ onto $\sd v = v$. 

To prove (ii), we choose $v$ in $\partial B$: Inside $\sd B$, the topological ball $|v^*|$  is triangulated as a cone with apex $v$ and basis $\sd \link_{B} v = \link_{\sd B} v$. The basis is collapsible by assumption. If we write down the list of faces removed in the collapse, and replace every face $\tau$ with $v \ast \tau$, we obtain a collapse of the cone onto its basis. (See also Theorem~\ref{thm:coneendo2}, part~(i).) Thus we may remove $|v^{\Diamond}|$ and $|v^*|$ via elementary collapses. We do this for each boundary vertex $v$. Next, we choose an edge $e$ of $\partial B$. The open dual block $|e^*|$ inside $\sd B$ is triangulated as a join of $e$ with a collapsible complex $K$, namely, the subdivision of the link of $e$. Reasoning as before, $e * K$ can be collapsed onto $K$, which topologically corresponds to the removal of $|e^{\Diamond}|$ and $|e^*|$. After doing this for all boundary edges, we pass to boundary $2$-faces, and so on. Eventually, we realize the cellular collapse of $B^*$ onto $C$ as a collapse of $\sd B$ onto~$\sd C$. 

\item By Corollary \ref{cor:DualityEC-C}, the complex $B^* - v^*$ cellularly collapses onto the dual of $\partial B$. Let $V$ be the subcomplex of $\sd B$ with the same underlying space of $B^* - v^*$. Our goal is to show that (i) $V$ collapses onto $\sd \partial B$ and (ii) $\sd B$ minus a facet collapses onto $V$. The proof of (i) is analogous to the proof of item (1), part (i). To prove part (ii), all we need to show is that the triangulation of $|v^*|$ inside $\sd B$ is endo-collapsible. This triangulation is precisely a cone with apex $v$ and basis $\sd \link_{B} v = \link_{\sd B} v$; the basis is endo-collapsible by assumption, so via Theorem~\ref{thm:cone} we are done.  
\end{compactenum}
\end{proof}

\section{Applications to classical combinatorial topology}

\subsection{Constructible and shellable manifolds} \label{sec:constructible}
In this section we study constructible complexes.  We prove that all constructible pseudo-manifolds are endo-collapsible (\textbf{Theorem \ref{thm:constructibleendo}}). Then we show that the barycentric subdivision of a constructible $d$-ball is collapsible 
(\textbf{Theorem \ref{thm:SDcollapsible}}).

\emph{Constructibility} is a strengthening of the Cohen-Macaulay notion, introduced by Hochster in the Seventies \cite{Hochster}. Recursively, a $d$-complex $C$ is constructible if and only if either $C$ has only one facet, or $C = C_1 \, \cup \, C_2$, where
\begin{compactenum}[(i)]
\item $C_1$ and $C_2$ are constructible $d$-complexes and
\item $C_1 \, \cap \, C_2$ is a constructible $(d-1)$-complex.
\end{compactenum}

\noindent The definition of constructibility may be strengthened by specifying that $C_2$ should be a simplex. The notion arising this way is called \emph{shellability} and has been extensively studied. Shellable complexes can be assembled one simplex at the time, so that each new simplex is attached to the previous ones alongside some (or maybe all) of its boundary facets. All shellable complexes are constructible, but some constructible $2$-complexes \cite{Htocome} and some constructible $3$-balls \cite{LutzEG} are not shellable. Also, it is easy to show that all shellable balls are collapsible \cite{DONG}, but the question whether all constructible $d$-balls are collapsible is open for each $d \ge 4$ \cite{BZ}.

Newman showed that all $2$-balls are shellable \cite{Newman1}; in particular, they are all constructible and collapsible. However, some $3$-balls are not collapsible \cite{BING}. Every constructible $3$-ball $B$ with more than one facet is  \emph{splittable}, that is, it contains an embedded $2$-disk $D$ such that $\partial D \subset \partial B$. Lickorish \cite{LICK1} showed that some $3$-balls do not split. By taking cones over Lickorish's example, one obtains non-constructible $d$-balls for each $d \ge 3$. 
The reason for the existence of non-constructible and non-collapsible balls goes back to classical knot theory; see Section \ref{sec:knots}. By \emph{knotted} manifolds we mean $3$-manifolds containing a non-trivial $3$-edge knot in their $1$-skeleton. Using a simple induction on the number of facets, Hachimori and Ziegler \cite{HZ} showed that no constructible $3$-manifold is knotted. On the other hand, we will see in Lemma \ref{lem:FurchBing} that a $3$-sphere may contain an arbitrarily complicated $3$-edge knot in its $1$-skeleton. Removing a $3$-cell from one such sphere, we obtain a knotted $3$-ball with the same complicated $3$-edge knot. All the examples produced this way are Cohen-Macaulay complexes that are not constructible.  

Summing up, for $d$-manifolds (with or without boundary)
\begin{equation} \label{eq:1}
\textrm{shellable} \, \Longrightarrow \,
\textrm{constructible} \, \Longrightarrow \,
\textrm{Cohen-Macaulay} \, \Longrightarrow \,
\textrm{strongly connected}, 
\end{equation}
\noindent and all implications are strict for $d \ge 3$.

The next theorem strengthens of a result of the author and Ziegler~\cite[Lemma 2.23]{BZ}, who showed that constructible pseudo-manifolds are LC, and a result of Chari \cite[Prop.~3]{Chari}, who showed that all shellable pseudomanifolds without boundary are endo-collapsible. 

\begin{thm} \label{thm:constructibleendo}
All constructible pseudo-manifolds are endo-collapsible.
\end{thm}

\begin{proof}
If $M$ has only one facet $\Delta$, then $M - \Delta$ \emph{is} the boundary of $M$ and there is nothing to prove. Otherwise, write $M = C_1 \cup C_2$, so that $C_1$, $C_2$ are constructible $d$-complexes with a constructible $(d-1)$-dimensional intersection. Since every $(d-1)$-face of $M$ is contained in at most two facets of $M$, every $(d-1)$-face of $C_1$ is contained in at most two facets of $C_1$; the same holds for $C_2$, of course. Also, every facet of $C_1 \cap C_2$  is contained in exactly one facet of $C_1$ and in exactly one facet of $C_2$. Therefore, the $C_i$ are constructible pseudo-manifolds with non-empty boundary (because $C_1 \cap C_2$ is contained in the boundary of each $C_i$). 
By a result of Zeeman \cite{Zeeman}, both the $C_i$ are $d$-balls. $C_1 \cap C_2$ is either a ball or a sphere,  depending on whether $M$ is a ball or a sphere.
By induction, we can assume that $C_1$, $C_2$ and $C_1 \cap C_2$ are endo-collapsible.
Applying Theorem \ref{thm:DepthUnion}, we conclude that $C$ is also endo-collapsible.
\end{proof}

\begin{remark}
The converse of Theorem \ref{thm:constructibleendo} is false: Some spheres and balls are endo-collapsible but not constructible. In the next section, we will show that endo-collapsible manifolds may be knotted, unless the knot is ``too complicated'' (cf.~Theorem~\ref{thm:EndoNoKnot}). In contrast, all knotted manifolds are not constructible.
\end{remark}

\begin{corollary} \label{cor:ShellableBall}
If $B$ is a constructible $d$-ball, then $B$ collapses onto its boundary minus a facet.
\end{corollary}

\begin{remark}
Corollary \ref{cor:ShellableBall} shows that any shellable ball admits \emph{two} natural  collapsing sequences, one collapsing the ball onto a single vertex, the other one collapsing the ball onto its boundary minus a facet. In general, these two sequences are ``not compatible'': In fact, when $d \ge 4$ any PL $(d-1)$-sphere (including the ones generated in Theorem \ref{thm:Obstruction}) is combinatorially equivalent to the boundary of some shellable $d$-ball \cite[Theorem~2, p.~79]{Pachner0}. So if $B$ is a shellable $d$-ball, $d \ge 4$, and $\sigma$ is a $(d-1)$-cell  in $\partial B$, we have that (1) $B$ always collapses onto $\partial B - \sigma$, (2) $B$ always collapses onto a point, but (3) ``often'' $\partial B - \sigma$  does not collapse onto a point. (The adverb ``often'' can be made precise by asymptotically counting the combinatorial types with respect to the number of facets: See Section~\ref{sec:LC}.)
\end{remark}

\begin{corollary}\label{cor:shellablesphere1}
If $S$ is a shellable sphere, then $S - \Delta$ is collapsible for \emph{all} facets $\Delta$.
\end{corollary}

\begin{remark}
If $S$ is a shellable sphere and $\Delta$ is the last facet of some shelling, then $S - \Delta$ is \emph{shellable}. Hachimori and Ziegler  \cite{HZ} asked whether the shellability of $S$ implies the shellability of $S - \Delta$ for \emph{all} facets $\Delta$. In other words, is any facet of a shellable sphere the last facet of some shelling? This problem is still open. (The answer is ``yes'' for boundaries of polytopes: See \cite[p.~243]{Z}.) Corollary \ref{cor:shellablesphere1} could be viewed as a step towards a positive answer, since all shellable balls are collapsible. 
\end{remark}

\begin{example}
Rudin's ball \cite{Rudin, Wotzlaw} is a non-shellable geometric subdivision of a tetrahedron. Chilling\-worth \cite{CHIL} showed that any geometric subdivision of a convex $3$-polytope collapses onto its boundary minus a facet. In particular, Rudin's ball is endo-collapsible. In fact, it is also constructible~\cite{PB}.  
\end{example}

If a $d$-sphere $S$ splits into two endo-collapsible balls that intersect in an endo-collapsible $(d-1)$-sphere, by Theorem \ref{thm:DepthUnion} $S$ is endo-collapsible. The next Lemma shows that the requirement on the intersection is no longer needed, if we know that one of the two $3$-balls is collapsible:

\begin{lemma} \label{lem:coneendo1}
Let $S$ be a $d$-sphere that splits into two $d$-balls $B_1$ and $B_2$, so that $B_1 \cap B_2 = \partial B_1 = \partial B_2$. \\If $B_1$ is endo-collapsible and $B_2$ is \emph{collapsible}, then $S$ is endo-collapsible.
\end{lemma}

\begin{proof}
Let $\Sigma$ be a facet of $B_1$. The collapse of $B_1 - \Sigma$ onto $\partial B_1$ can be read off as a collapse of $S - \Sigma$ onto $B_2$. Yet $B_2$ is collapsible, so $S - \Sigma$ is collapsible.
\end{proof}

\begin{thm} (Cones and suspensions) \label{thm:coneendo2} 
\begin{compactenum}[\rm(i)]
\item If $B$ is an arbitrary ball, $v \ast B$ is a collapsible ball. If $B$ is collapsible, $v \ast B$ collapses onto $B$.
\item A manifold $M$ is endo-collapsible if and only if $v \ast M$ is endo-collapsible.
\item If $M$ is an endo-collapsible manifold, so is its suspension.
\item If $B$ is an endo-collapsible $d$-ball, $\partial (v \ast B)$ is an endo-collapsible $d$-sphere.
\item If $B$ is a collapsible $d$-ball and $\partial B$ is an endo-collapsible $(d-1)$-sphere, $\partial (v \ast B)$ is an endo-collapsible $d$-sphere. 
\item If $B$ is an endo-collapsible $d$-ball and $\partial B$ is an endo-collapsible $(d-1)$-sphere, $B$ is collapsible.
\item If $B$ is an endo-collapsible $(d-1)$-ball contained in an endo-collapsible $(d-1)$-sphere $A$, the $d$-ball $v \ast A$, where $v$ is a new vertex, collapses onto $A - B$. 
\end{compactenum}
\end{thm}

\begin{proof}
Item (i) is well known, see e.g. Welker \cite{Welker}. 
Item (ii) follows directly from Theorem \ref{thm:cone}. 

Item (iii) follows from item (ii) and Theorem \ref{thm:DepthUnion} (or alternatively, from item~(ii), Lemma~\ref{lem:coneendo1} and the fact that all cones are collapsible). 

Item (iv) follows from Lemma~\ref{lem:coneendo1}: The sphere $\partial(v\ast B)$ splits as the union of the ball $v \ast \partial B$ (which is collapsible) and  the ball $B$ (which is endo-collapsible) with $(v \ast \partial B) \cap B = \partial B = \partial (v \ast \partial B)$. 
Item (v) is very similar: It follows from Lemma~\ref{lem:coneendo1}, but this time we use that $v \ast \partial B$ is endo-collapsible and $B$ is collapsible.

As for item (vi), let $\sigma$ be a $(d-1)$-face of $\partial B$ and $\Sigma$ the unique $d$-cell containing $\sigma$. The collapse of $B - \Sigma$ onto $\partial B$ can be rewritten as a collapse of $B$ onto $\partial B - \sigma$. 
If the latter complex is collapsible onto a point, so is $B$. This settles (vi).

We are left with (vii). Let $\sigma$ be a facet of $B$. Since $A$ is endo-collapsible,  
by item (ii) $v \ast A$ is endo-collapsible, so $v \ast A - v \ast \sigma$ collapses onto $\partial (v \ast A) = A$. Equivalently, $v \ast A - v \ast \sigma - \sigma$ collapses onto $A - \sigma$. Since the removal of $(\sigma, v \ast \sigma)$ from $v \ast A$ is an elementary collapse, we conclude that $v \ast A$ collapses onto $A - \sigma$. But by the endo-collapsibility of $B$, $B - \sigma$ collapses onto $\partial B$. In particular, $A - \sigma$ collapses onto $A - B$.  
\end{proof}

\begin{corollary}
Let $v$ be a vertex of a $d$-sphere $S$. 
\begin{compactenum}[\rm (i)]
\item If the deletion  $\del_S \, v$ is an endo-collapsible $d$-ball, then $S$ is an endo-collapsible $d$-sphere. 
\item If the deletion  $\del_S \, v$ is a collapsible $d$-ball and in addition $\link_S \, v$ is endo-collapsible, then $S$ is an endo-collapsible $d$-sphere. 
\end{compactenum}
\end{corollary}

\begin{proof}
Set $B :=\del_S \, v$. Since $\partial B = \link_S \, v$, we have $\partial (v \ast B) = B \cup (v \ast \partial B) = B \cup (v \ast \link_S \, v) = S$. The first item follows then from Theorem~\ref{thm:coneendo2}, item (iv); the second, from Theorem~\ref{thm:coneendo2}, item (v).
\end{proof}

The converse of implication (i) above does not hold: A counterexample is given by Lutz's sphere $S^3_{13,56}$ (see the proof of Corollary~\ref{cor:EndoNoKnot2}). This sphere is endo-collapsible, but deleting one of the three vertices of the knot, one obtains a $3$-ball that cannot be endo-collapsible by Corollary~\ref{cor:EndoNoKnot2}.

\begin{example} \label{ex:polytope}
Let $S$ be the boundary of a $(d+1)$-polytope $P$. Then:
\begin{compactenum}[(1)]
\item $S$ is a shellable $d$-sphere \cite[p.~240]{Z};
\item $S^*$ is also a shellable $d$-sphere (it is the boundary of the dual polytope of $P$);
\item $\link_S \, v$ is a shellable $(d-1)$-sphere, for each vertex $v$ of $S$ \cite[Ex.~8.6, p.~282]{Z};
\item $S - \Delta$ is a shellable $d$-ball, for each facet $\Delta$ of $S$ \cite[Corollary~8.13]{Z};
\item $\del_S \, v$ is a shellable $d$-ball, for each vertex $v$ of $S$ \cite[Corollary~8.13]{Z}.
\end{compactenum} 
\noindent By Theorem \ref{thm:constructibleendo}, $S$, $S^*$, $\link_S \, v$, $\del_S \, v$, and $S-\Delta$ are all endo-collapsible. The last two complexes are also collapsible. 
\end{example}

Since constructibility and collapsibility are classical properties of contractible simplicial complexes, it is natural to ask whether all constructible balls are collapsible. (This was asked also by Hachimori \cite{Hthesis} and in our work with Ziegler \cite{BZ}). The problem is open for all dimensions greater than three. Here we present a partial positive answer for all dimensions.

\begin{lemma} \label{lem:linkConstructible}
The link of any non-empty face in a constructible ball is constructible. 
\end{lemma}

\begin{proof}
Let $C$ be a constructible $d$-ball. We proceed by double induction, on the dimension and the number of facets. If $C$ has only one facet, then $C$ is a polytope. The link of a face $\sigma$ in $C$ is just the deletion of $\sigma$ from $\partial C$. Now, the boundary of a polytope can be shelled starting at the star of any face $\sigma$. (See \cite[p.~243]{Z} for the case in which $\sigma$ is a vertex). In particular, deleting $\sigma$ from $\partial C$ yields a shellable ball. So the link of $\sigma$ in $C$ is constructible and we are done.
If $C$ has more than one facet, then $C=C_1 \cup C_2$. By the inductive assumption, the three complexes  $\link_{C_1} \sigma$, $\link_{C_2} \sigma$ and $\link_{C_1 \cap C_2} \sigma$ are constructible. Then $\link_C \sigma = \, \link_{C_1} \sigma \, \cup \, \link_{C_2} \sigma$, with 
$\link_{C_1} \sigma
\, \cap \,
\link_{C_2} \sigma
\, = \,
\link_{C_1 \cap C_2} \sigma$.
\end{proof}

\begin{thm} \label{thm:SDcollapsible}
The barycentric subdivision of a constructible ball is collapsible. 
\end{thm}

\begin{proof} We proceed by induction on the dimension: The case $d=2$ is trivial, since all $2$-balls are collapsible. The case $d=3$ was proven in \cite{BZ}. Let $B$ be a constructible $d$-ball, $d \ge 4$. By Theorem \ref{thm:constructibleendo}, $B$ is PL endo-collapsible. If $\delta$ is any boundary face of $B$, then $\link_B \delta$ is  constructible by Lemma \ref{lem:linkConstructible} and has lower dimension than $B$. By the inductive assumption, $\sd \link_B \delta$ is collapsible. Moreover, for each non-empty face $\sigma$ of $B$,  $\sd \link_B \sigma$ is constructible by Lemma \ref{lem:linkConstructible} and thus PL endo-collapsible by Theorem \ref{thm:constructibleendo}.  Via Corollary \ref{cor:subdivision1}, part (1), we conclude that $\sd B$ is collapsible.
\end{proof}
 
\begin{remark} \label{rem:Pachner}
For the previous result we did not require $\partial B$ to be triangulated ``nicely'': A priori, if $\delta$ is a boundary facet of $B$, the complex $\link_{\partial B} \delta$ might not be endo-collapsible. In fact, in the next section we will show how to produce a simplicial $3$-sphere $S$ with a $3$-edge knot corresponding to the connected sum of $6$ trefoils in its $1$-skeleton. Via Theorem \ref{thm:Obstruction} we will see that neither the barycentric subdivision $\sd S$ nor the first suspension $S'= (v \ast S) \cup(w \ast S)$ can be LC. Now, if $B$ is a shellable $5$-ball with $\partial B = S'$ (cf.~Pachner \cite[Theorem~2, p.~79]{Pachner0}), this $B$ is constructible and $\sd B$ is collapsible. However, $\sd \link_{\partial B} v = \sd \link_{S'} v = \sd S$ is not LC.
\end{remark}

\begin{remark}
Is every link in a shellable complex shellable? We know the answer is positive for simplicial or cubical complexes \cite{Cou}, but in general this still seems to be an open problem. 
\end{remark}

\subsection{Knots inside 3-manifolds} \label{sec:knots}
A \emph{spanning edge} in a $3$-ball $B$ is an interior edge
that has both endpoints on $\partial B$. A \emph{$K$-knotted
spanning edge} of a $3$-ball $B$ is a spanning edge $[x, y]$ such that some
simple path on $\partial B$ between $x$ and $y$ completes the edge to a
non-trivial knot $K$.  (\emph{Non-trivial} means that it does not bound a disk; all the knots we consider are \emph{tame}, that is, realizable as $1$-dimensional subcomplexes of some triangulated $3$-sphere.) The knot type does not depend on the boundary path chosen; so, the knot is determined by the edge. By \emph{knot group} we mean the fundamental group of the knot complement inside the $3$-sphere. For further definitions, e.g.~of \emph{connected sums} of knots, of \emph{$2$-bridge} knots, of \emph{spanning arcs} and so on, we refer either to Kawauchi \cite{KAWA} or to our paper with Ziegler \cite{BZ}. ``LC $3$-balls'' and ``endo-collapsible $3$-balls'' are the same. (The proof is elementary, see e.g.~Cor.~\ref{cor:poincare1}.)

In this section, we prove that endo-collapsible $3$-balls cannot contain any knotted spanning edge (\textbf{Corollary ~\ref{cor:EndoNoKnot2}}). This improves a result of the author and Ziegler, who showed in \cite[Proposition~3.19]{BZ} that endo-collapsible $3$-balls \emph{without interior vertices} cannot contain any knotted spanning edges. Our proof is not an extension of the proof of \cite[Section~3.2]{BZ}, but rather an application of Theorem~\ref{thm:RelativeObstruction}.
\enlargethispage{4mm}
 \begin{figure}[htbf]
\begin{center}
\includegraphics[width=.32\linewidth]{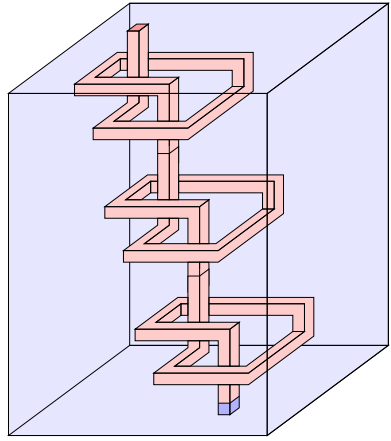} 
\caption{How to build a $3$-ball $B$ with a $K$-knotted spanning edge: (1) Take a pile of cubes and drill a tubular hole  downward shaped as $K$, (2) stop one edge before reaching the bottom. (Here $K$ is the connected sum of three trefoil knots.) }
  \label{fig:nonLC}
\end{center}
\end{figure}

Corollary~\ref{cor:EndoNoKnot2} establishes a concrete difference between collapsibility and endo-collapsibility. In fact, some collapsible $3$-balls have knotted spanning edges \cite{BZ, HAM, LICKMAR}. We also show that the barycentric subdivision of a collapsible $3$-ball is endo-collapsible (\textbf{Proposition~\ref{prop:subdivideC-EC}}). In particular, we re-cover a result by Bing: A ball with a knotted spanning edge cannot be collapsible if the knot is the double trefoil. In fact, its barycentric subdivision is not endo-collapsible by \textbf{Theorem~\ref{thm:EndoNoKnot}}.

\begin{lemma}[Furch {\cite{FUR}}, Bing {\cite{BING}}] \label{lem:FurchBing}
Any (tame) knot can be realized as \emph{$3$-edge} subcomplex of some triangulated $3$-sphere, and as knotted spanning edge of some triangulated $3$-ball.
\end{lemma}

\begin{proof}
Let $K$ be an arbitrary knot. Take a sufficiently large pile of cubes, start from the top and drill a tubular hole downward along $K$. Stop one step before perforating the $3$-ball. The result is a $3$-ball with a spanning edge $[x,y]$. (See Figure \ref{fig:nonLC}.) Any boundary path $\wp$ from $x$ to $y$ has to climb up the tube, thus completing the edge $[x,y]$ to a knotted loop, where the knot is isotopic to $K$. Thus we have realized $K$ as knotted spanning edge in a $3$-ball $B$. Now pick a new vertex $v$: The $3$-sphere  $S_B = \partial (v \ast B)$ contains the subcomplex $[x,y] \cup [y, v] \cup [x, v]$, which is isotopic to $[x, y] \cup \wp$ inside $S_B$.   
\end{proof}

\begin{lemma}[Goodrick {\cite{GOO}}] \label{lem:Goodrick}
If $K$ is the connected sum of $k$ trefoil knots embedded in a $3$-sphere $S$, then the fundamental group of $|S|-|K|$ admits no presentation with fewer than $k+1$ generators.
\end{lemma}

\begin{thm}[{\cite[Cor.~5.5.5, p.~112]{Benedetti-diss}}] \label{thm:Obstruction}
Let $k$ and $d$ be two arbitrary integers, with $d \ge 3$. 
\begin{compactenum}[\rm(1)]
\item There exists a PL $d$-sphere $M$ on which \emph{any} discrete Morse function has more than $k$ critical $(d-1)$-cells.  
\item There exists a PL $d$-ball $B$ on which \emph{any} boundary-critical discrete Morse function has more than $k$ critical interior $(d-1)$-cells.  
\end{compactenum}
\end{thm}

\begin{proof} 
By Lemma \ref{lem:FurchBing}, there is a $3$-sphere $S$ with a $3$-edge knot $K$ in its $1$-skeleton, where the knot is the sum of $k+3$ trefoils. By Lemma~\ref{lem:Goodrick}, every presentation of $\pi_1(|S|-|K|)$ has at least $(k+3)+1$ generators.
Suppose that $S$ admits a discrete Morse function with $c_2$ critical $2$-cells. By Theorem \ref{thm:RelativeObstruction}, there is a connected cell complex with $c_2+3$ edges that is homotopy equivalent to $|S|-|K|$. In particular, the fundamental group of $|S|-|K|$ admits a presentation with $c_2 + 3$ generators, so $c_2 + 3 \ge k + 3 + 1$.   

Higher-dimensional examples are obtained by suspensions. Suspending a $3$-sphere $S$ with a $3$-edge knot $K$ creates a $4$-sphere $S'$ with a $2$-subcomplex $K'$ with $6$ triangles; topologically, $|S'|-|K'|$ is homotopy equivalent to $|S|-|K|$. 
Suppose $S$ contains a $3$-edge knot which is the sum of $t$
trefoils. The $(d-3)$-rd suspension of $S$ is a $d$-sphere $S'$ that contains a subcomplex $K'$  with $3 \cdot 2^{d-3}$ facets. Assume this $S'$ admits a discrete Morse function with $c_{d-1}$ critical $(d-1)$-cells. By Theorem \ref{thm:RelativeObstruction}, there is a connected cell complex with exactly $(c_{d-1}+3 \cdot 2^{d-3})$ loops that is homotopy equivalent to $|S'|-|K'|$, or equivalently to $|S|-|K|$. In particular, the fundamental group of $|S|-|K|$ admits a presentation with $c_{d-1} + 3 \cdot 2^{d-3}$ generators. Since $|S|-|K|$ is the knot group of the connected sum of $t$ trefoil knots, $t < c_{d-1} + 3 \cdot 2^{d-3}$ by Lemma~\ref{lem:Goodrick}. In other words, if $t \ge  c_{d-1} + 3 \cdot 2^{d-3}$ the PL $d$-sphere $S'$ admits no discrete Morse function with $c_{d-1}$ critical $(d-1)$-cells. Setting $t := k + 3 \cdot 2^{d-3}$ we conclude.

As far as (2) is concerned, take a $3$-ball $B$ with a $K$-knotted spanning edge, where $K$ is the sum of $k+1$ trefoil knots. The fundamental group of $|B| - |K^{\int}|$ coincides with the knot group. By Lemma \ref{lem:Goodrick}, any presentation of the fundamental group of $|B| - |K^{\int}|$ has at least $(k+1) + 1$ generators. Suppose  $B$ admits a boundary-critical discrete Morse function on $B$ with $c_2^{\int}$ critical $2$-cells. Note that $K^{\int}$ consists of one single edge. Theorem \ref{thm:RelativeObstruction} yields a presentation of the fundamental group of $|B| - |K^{\int}|$ with $c_2^{\int} + 1$ generators. Therefore, $c_2^{\int} + 1 \ge k+1+1$. Higher-dimensional examples are obtained by taking cones over the previous example, cf.~Theorem~\ref{thm:cone}. Since taking cones does not increase the number of facets, there is no need to consider more complicated knots as the dimension increases. 
\end{proof}

\begin{thm} \label{thm:EndoNoKnot}
Let $K$ be any knot whose group admits no presentation with fewer than $k+1$ generators. (For example, the connected sum of $k$ trefoil knots.) Then every $3$-ball with a $K$-knotted spanning arc of $e$ edges admits no boundary-critical discrete Morse function with $k-e$ critical interior $2$-cells.
\end{thm}

\begin{proof}
Let $B$ be a $3$-ball with a $K$-knotted spanning edge. Choose a discrete Morse function on $B$, with $c^{\int}_2$ critical interior $2$-cells, say. Let $[x,y]$ be the knotted spanning edge. Via Theorem \ref{thm:RelativeObstruction} we obtain a presentation with $c^{\int}_2 + e$ generators of the fundamental group of $|B| - |[x,y]|$. Yet the fundamental group of $|B| - |[x,y]|$ coincides with the knot group of $K$. By the assumption, $c^{\int}_2 + e \ge k+1$.
\end{proof}

\begin{corollary} \label{cor:EndoNoKnot1}
Let $r$ be a non-negative integer. Let $K$ be the sum of $2^r$ trefoils. If a $3$-ball $B$ contains a $K$-knotted spanning edge, then $\sd^r B$ is not endo-collapsible.
\end{corollary}

\begin{proof}
Performing a single barycentric subdivision divides each edge into two. Thus the $3$-ball $\sd^r B$ has a $K$-knotted spanning arc consisting of $2^r$ edges. By Theorem~\ref{thm:EndoNoKnot}, $\sd^r B$ cannot admit a boundary-critical discrete Morse function without critical interior $2$-cells.
\end{proof}

\begin{corollary} \label{cor:EndoNoKnot2}
Endo-collapsible $3$-balls cannot contain knotted spanning edges. However, some endo-collapsible $3$-balls and $3$-spheres contain $3$-edge knots in their $1$-skeleton. 
\end{corollary}

\begin{proof}
The first part is the case $r=0$ of Corollary \ref{cor:EndoNoKnot1}. For the second part, consider Lutz's sphere $S^3_{13, 56}$ with $13$ vertices and $56$ tetrahedra \cite{LUTZ1}. This simplicial $3$-sphere contains a trefoil knot on $3$~edges. Let $B_{13, 55}$ be the $3$-ball obtained by removing from $S^3_{13, 56}$ the tetrahedron $\Delta=\{1, 2, 6, 9\}$. With the help of a computer program by Lutz, we showed that $B^3_{13, 56}$ minus the facet $\Sigma = \{2,6,9,11\}$ collapses onto the boundary of $B^3_{13, 56}$, which is just the boundary of $\Delta$. (See \cite[Section~5.4.1]{Benedetti-diss}.)  So $B^3_{13, 56}$ is both collapsible and endo-collapsible. In particular, the $3$-sphere $S^3_{13, 56}$ is endo-collapsible. It is neither shellable nor constructible, because of the presence of a $3$-edge knot~\cite{HZ}. 
\end{proof} 

\begin{corollary} \label{cor:StrictlyC-EC}
For $3$-balls, endo-collapsibility strictly implies collapsibility. The endo-collapsi\-bility of a manifold with boundary $M$ strictly implies the collapsibility of the dual block complex $M^*$.
\end{corollary}

\begin{proof}
All $2$-spheres and all $2$-balls are endo-collapsible. (This can be proven either directly, or using Theorem \ref{thm:constructibleendo}.) Let $\sigma$ be a $2$-cell in the boundary of an endo-collapsible $3$-ball $B$; let $\Sigma$ be the unique $3$-cell containing $\sigma$. A collapse of $B - \Sigma$ onto $\partial B$ can be seen as a collapse of $B - \Sigma - \sigma$ onto $\partial B - \sigma$, and the other way around. Since $\sigma$ is a free face of $\Sigma$, we conclude that $B$ collapses onto $\partial B - \sigma$, which is a $2$-ball and thus collapsible. So, all endo-collapsible $3$-balls are collapsible. Now, choose a collapsible $3$-ball $B$ with a knotted spanning edge \cite{BZ, LICKMAR}. (The knot can be the trefoil, or any $2$-bridge knot.) In view of Corollary~\ref{cor:EndoNoKnot2}, $B$ is collapsible but not endo-collapsible. Thus we can apply Corollary \ref{cor:DualityEC-C}: $B^*$ is an endo-collapsible $3$-ball. In particular, $B^*$ is collapsible. 
\end{proof}

We conclude re-proving a classical result by Bing \cite{BING}, who first described a non-collapsible ball by considering a $3$-ball with a double-trefoil-knotted spanning edge:

\begin{prop} \label{prop:subdivideC-EC}
The barycentric subdivision of a collapsible $3$-ball is endo-collapsible. 
\end{prop}
\enlargethispage{3mm}
\begin{proof}
All $3$-balls are PL. The link of any non-empty face in a $3$-ball is either a $2$-ball or a $2$-sphere and thus endo-collapsible. Via Corollary \ref{cor:subdivision1}, part~(2), if $B$ is collapsible $\sd B$ is endo-collapsible.
\end{proof}

\begin{remark} \label{rem:subdivideC-nonC}
The previous result is best possible, in the sense that the barycentric subdivision of a collapsible $3$-ball need not be constructible. To see this, take any collapsible $3$-ball with a knotted spanning edge (cf. Corollary~\ref{cor:sharp} below), perform the barycentric subdivision and apply \cite[Lemma~1]{HZ}.
\end{remark}

\begin{corollary}[essentially Bing] \label{cor:Goodrick}
Let $K$ be the connected sum of $k$ trefoil knots, or any knot whose group admits no presentation with fewer than $k+1$ generators. Let $B$ be a $3$-ball with a $K$-knotted spanning arc of $e$ edges. If $k \ge 2e$, then $B$ cannot be collapsible.
\end{corollary}

\begin{proof}
Consider the barycentric subdivision $\sd B$ of $B$. This is a $3$-ball with a $K$-knotted spanning arc of $2e$ edges. Since $k \ge 2e$, $\sd B$ cannot be endo-collapsible  by Theorem~\ref{thm:EndoNoKnot}. Via Proposition~\ref{prop:subdivideC-EC} we conclude that $B$ cannot be collapsible. 
\end{proof}

\begin{corollary} \label{cor:sharp}
Let $k$ be a positive integer and let $K$ be the connected sum of $k$ trefoil knots. A $3$-ball with a $K$-knotted spanning arc of $2k$ edges can be endo-collapsible. 
\end{corollary}

\begin{proof}
In \cite{BZ} we explained how to construct a collapsible $3$-ball $B_k$ with a $K$-knotted spanning arc of $k$ edges. (See also \cite{LICKMAR, HAM}.) By Proposition~\ref{prop:subdivideC-EC}, $\sd B_k$ is endo-collapsible. Yet $\sd B_k$ has a $K$-knotted spanning arc of $2k$ edges. 
\end{proof}

\subsection{Collapse depth versus ring-theoretic depth} \label{sec:cdepth}
In this section we relate what we have studied so far to classical notions in Commutative Algebra, e.g. Cohen--Macaulayness. We show that the collapse depth of a simplicial complex is never larger than the ring-theoretic depth of the associated Stanley-Reisner ring (\textbf{Theorem \ref{thm:Hd-cd}}). Sometimes the two depths coincide, but they can be arbitrarily far apart, as shown by \textbf{Remark~\ref{rem:nonLCspheres}} and \textbf{Proposition~\ref{prop:farapart}}. We also give a hierarchy of properties of triangulated manifolds, valid for any dimension (\textbf{Theorem \ref{thm:endo-hierarchy}}). 

The \emph{algebraic depth} of a simplicial complex  depends only on the homology of the underlying topological space. Intuitively, it measures how far we should go with the homology groups of a space and its links requiring them to be zero. The complexes of maximal algebraic depth are called \emph{Cohen--Macaulay}: All triangulated $d$-balls and $d$-spheres are Cohen--Macaulay, but manifolds with different topologies may be Cohen--Macaulay as well. For example, Miller and Reiner found a PL Cohen--Macaulay non-simply-connected $4$-manifold that embeds in $\mathbb{R}^4$ without being a $4$-ball \cite{MillerReiner}. 

Here we present two similar notions (dual to one another) 
that are equally interesting from the combinatorial point of view. The goal is to assign to each manifold an integer that (1) tells us about the vanishing of \emph{homotopy} groups (and not just of \emph{homology} groups), (2) takes into consideration how ``nicely'' $M$ is triangulated. For example, the integer should increase if the $(d-2)$-skeleton of contains complicated $(d-2)$-knots. We had already defined ``collapse depth'' in Section \ref{sec:background}, but for convenience we repeat the definition here. We called the other notion ``hamiltonian depth'': The name comes from the fact that a subcomplex $H$ of a complex $C$ is called $k$-hamiltonian if the $k$-skeleta of $H$ and $C$ coincide~\cite{EffenbergerKuehnel}.

\begin{definition}[Collapse depth]
The \emph{collapse depth} $\cd M$ of a pseudo-manifold $M$ is the maximal integer $k$ for which there exists a boundary-critical discrete Morse function $f$ on $M$ with one critical $d$-cell and no internal critical $(d-i)$-cells, for each $i \in \{1, \ldots, k-1\}$.
\end{definition}

\begin{definition}[Hamiltonian depth]
The \emph{hamiltonian depth} $\Hd C$ of a (regular CW) complex $C$ is the maximal integer $k$ for which there exists a $k$-subcomplex $H$ of $C$ that is $(k-1)$-hamiltonian and collapsible. 
\end{definition}

\begin{lemma}\label{lem:HdepthConnectivity}
Let $C$ be a (regular CW) complex. Let $k>1$ an integer. Then
\begin{compactenum}[\rm (i) ]
\item $\Hd C \ge 1$ if and only if $C$ is connected;
\item  $\Hd C \ge k$ if and only if some discrete Morse function $f$ on $C$ has one critical $0$-cell and no critical $i$-cells for each $i \in \{1, \ldots, k-1\}$.  
\end{compactenum}
In particular, if $\Hd C \ge k$ then $C$ is at least $(k-1)$-connected. 
\end{lemma}

\begin{proof} By definition, $\Hd C \ge 1$ means that there exists a collapsible $1$-complex $H$ that contains all vertices of $C$. Thus $\Hd C \ge 1$ if and only if $C$ admits a spanning tree. This means that $C$ is connected. Similarly, $\Hd C \ge k$ means that there exists a collapsible $k$-complex $H$ that contains all $(k-1)$-faces of $C$. Being collapsible, $H$ is contractible. The complex $C$ is obtained from $H$ by attaching $i$-cells, with $i \ge k$.  If $j \le k-1$, these attachments leave the $j$-th homotopy group unchanged; so $C$ is $(k-1)$-connected. Finally, since $H$ is a collapsible subcomplex of $C$, there is a discrete Morse function $h$ on $H$ with one critical $0$-cell. We can extend $g$ to a function $f$ on $M$ by setting $f \equiv g$ on $H$ and $f (\sigma) = \dim \sigma + c$ if $\sigma \in M - H$. If $c$ is chosen large enough, $f$ is a discrete Morse function. This  $f$ might have plenty of critical cells; however, $f$ has no critical cells of dimension smaller than $k$.  
\end{proof}

The locally constructible (LC) manifolds first studied by Durhuus and Jonsson \cite{DJ} are precisely the manifolds with collapse depth at least two (cf.~Theorem~\ref{thm:charLCpseudo}). The next Proposition relates the two notions of collapse depth and Hamiltonian depth to one another:  

\enlargethispage{4mm}
\begin{thm} \label{thm:Hd-cd}
For any $d$-manifold $M$, 
$1 \le \Hd M \le \ad M \le d$
and
$1 \le \cd M \le \ad M \le d$. Moreover:
\begin{compactenum}[\rm (i) ]
\item If $\Hd M \ge 2$ or $\cd M \ge 2$, then $M$ is simply connected.
\item If $\Hd M = d$,  then $M$ is Cohen--Macaulay and at least $(\dim M-1)$-connected. 

If in addition $\partial M = \emptyset$, then $M$ is a sphere.
\item If $\cd M = d$ then $M$ is either a sphere or a ball. 

In particular, $M$ is Cohen--Macaulay and at least $(d-1)$-connected.
\item If $M$ is PL, then $\cd M \le \Hd M^*$, $\Hd M \le \cd M^*$ and  $M$ is $(\cd M -1)$-connected.

If in addition $\partial M = \emptyset$, then $\cd M = \Hd M^*$.
\end{compactenum}
\end{thm}

\begin{proof}
Fix a $d$-manifold $M$ and set  $c = \cd M$, $h = \Hd M$. Since $M$ is a manifold, all links inside $M$ are homology spheres or homology balls. By Lemma~\ref{lem:HdepthConnectivity} and the Morse inequalities, all homology groups of $M$ vanish up to the $u$-th one (and possibly even further). So $\ad M \ge h$ by definition. Analogously, from Corollary~\ref{cor:MorseInequalities} we get that  $\ad M \ge c$. The inequalities $1 \le c \le d$  and $1 \le h \le d$ are obvious. By Lemma~\ref{lem:HdepthConnectivity}, if $h \ge 2$ then $M$ is $1$-connected. The same holds 
if $c \ge 2$, but the proof is non-trivial: See Lemma~\ref{lem:SimplyConnected}. This settles (i). 
Simply connected manifolds are orientable. Item (ii) is then a direct consequence of Lemma~\ref{lem:HdepthConnectivity}. (If $\partial M = \emptyset$ the 
polarity principle \cite[Theorem~4.6]{FormanADV} and Forman's sphere theorem (Theorem~\ref{thm:SphereTheorem}) imply that $M$ is an endo-collapsible sphere.)
Finally, item (iii) is a re-wording of our Ball Theorem~\ref{thm:FirstBall} and item (iv) is a corollary of Theorem \ref{thm:relativeduality}. 
\end{proof}

\begin{remark} \label{rem:nonLCspheres}
Theorem~\ref{thm:Hd-cd} does not extend to pseudo-manifolds. In fact, a pinched annulus has maximal collapse depth without being Cohen--Macaulay nor simply connected. Moreover, not all simply connected triangulated manifolds have $\cd \ge 2$: In fact, for each $d \ge 3$ the author and Ziegler constructed a PL $d$-sphere $S$ with $\cd S = 1$ \cite{BZ}. (See also Theorem \ref{thm:Obstruction}.) The dual block decomposition of this $S$ yields a $d$-sphere $S^*$ with $\Hd S = 1$.
\end{remark}

\begin{corollary} \label{cor:poincare}
Let $d \ge 3$ be an integer. Let $H$ be a manifold that is a homology $d$-pseudo-sphere but not a sphere. Then $\cd H = \Hd H = 1$ and $\ad H = \dim H = d$.
\end{corollary}

\begin{proof}
Since the algebraic depth does not distinguish between spheres and homology spheres, $H$ is Cohen--Macaulay. Were $H$ simply connected, by the Poincar\'{e} conjecture $H$ would be a sphere, a contradiction. So $H$ is not $1$-connected. By Theorem~\ref{thm:Hd-cd}, item (i), both the collapse depth of $H$ and the Hamiltonian depth of $H$ are smaller than two. 
\end{proof}

\begin{corollary}
For each $d$ even, $d > 1$, there exists a $d$-manifold whose collapse depth equals $\frac{d}{2}$.
\end{corollary}

\begin{proof}
Let $S_k$ be an endo-collapsible PL triangulation of a $k$-sphere. (For example, the boundary of 
 the $(k+1)$-simplex or of the $(k+1)$-cube.) By \cite[Proposition~2.3]{B} the polytopal complex $S_k \times S_k$, after the removal of a facet, collapses onto two copies of $S_k$ glued together at a single vertex. This implies $\cd (S_k \times S_k) \ge k$. Yet $(S_k \times S_k)$ is not $k$-connected: Via Theorem~\ref{thm:Hd-cd}, item (iv), this implies $\cd (S_k \times S_k) < k+1$. So $\cd (S_k \times S_k) = k$.
\end{proof}

\begin{corollary} \label{cor:poincare1}
For $d \ge 3$, no $d$-ball, no $d$-sphere and no PL $d$-manifold without boundary has collapse depth $d-1$. 
\end{corollary}

\begin{proof}
If $S$ is a sphere (not necessarily PL) and $\cd S \ge d-1$ then $S$ minus a facet collapses onto a contractible $1$-complex $C$. Since all contractible $1$-complexes are collapsible, $\cd S = d$. Similarly, if $B$ is a ball and $\cd B \ge d-1$ then $B$ minus a facet collapses onto the union of $\partial B$ with some trees. Therefore, $B$ minus a facet can be collapsed onto $\partial B$. 

Now, let $M$ be a PL manifold \emph{without boundary}. If $\cd M \ge d-1$, by Theorem~\ref{thm:Hd-cd}, part (iv), we can say that $M$ is $(d-2)$-connected. Every $(d-2)$-connected manifold without boundary $M$ is a homotopy sphere: In fact, via Poincar\'{e} duality and the Universal Coefficient Theorem for Cohomology one can show that $\pi_{d-1} (M) = H_{d-1} (M) = 0$ and $\pi_{d} (M) = H_{d} (M) = \mathbb{Z}$. By the Poincar\'{e} conjecture, $M$ is a sphere and we are back in a previously studied case.
\end{proof}

\noindent In contrast, some PL $d$-manifolds with boundary are $(d-2)$-connected but not $(d-1)$-connected. In fact, Corollary~\ref{cor:poincare1} does not extend to manifolds with boundary: A $3 \times 3 \times 3$ pile of cubes with the central cube removed yields a $3$-manifold with collapse depth two.

Collapse depth and Hamiltonian depth may be arbitrarily far apart from the algebraic depth, as shown in Remark~\ref{rem:nonLCspheres}. The next Proposition shows that collapse depth and Hamiltonian depth may also be arbitrarily far apart from one another:

\begin{prop} \label{prop:farapart}
For each $d \ge 3$, there is a PL $d$-ball $B$ with $\cd B = 1$ and $\Hd B = \ad B = d$. 
\end{prop}

\begin{proof}
It suffices to prove the claim for $d=3$: In fact, given a $d$-ball $B$ and a new vertex $v$, 
we have $\cd (v \ast B) = \cd (B)$ by Theorem \ref{thm:cone} and $\Hd (v \ast B) = d+1$ by Theorem~\ref{thm:coneendo2}, item (i). Now, take a collapsible $3$-ball $B$ with a knotted spanning edge. By Corollary~\ref{cor:EndoNoKnot2}, $\cd B < 3$; by Corollary \ref{cor:poincare1}, $\cd B \neq 2$. Thus $\cd B = 1$. However, $\Hd B = \dim B = 3$ because $B$ is collapsible.
\end{proof}

In the paper~\cite{Benedetti-DMTapaMT} we show how the collapse depth of a manifold is deeply related to its geometrical connectivity, introduced by Wall \cite{Wall}.  We conclude the present work with a hierarchy of some of the properties we studied so far:

\begin{thm} \label{thm:endo-hierarchy}
For $d$-manifolds ($d \ge 2$),
{\em
\[
\textrm{shellable} \, \Longrightarrow \,
\textrm{constructible} \, \Longrightarrow \,
\textrm{endo-collapsible} \, \Longrightarrow \,
\textrm{LC} \, \Longrightarrow \,
\textrm{simply connected}
\]}  
and
{\em
$\textrm{endo-collapsible} \, \Longrightarrow \, \textrm{Cohen--Macaulay} .
$}
\noindent For $d \ge 3$, all these implications are strict. However, LC $3$-manifolds without boundary and endo-collapsible $3$-manifolds without boundary are the same. 
\end{thm}

\begin{proof}
All shellable complexes are constructible. Constructible manifolds are endo-collapsible by Theorem~\ref{thm:constructibleendo}. To show that endo-collapsible manifolds are Cohen--Macaulay, use either Theorem~\ref{thm:Hd-cd} or Theorem \ref{thm:constructibleendo}. The implication ``endo-collapsible $\Rightarrow$ LC'' is a direct consequence of Theorem \ref{thm:charLCpseudo}. Finally, LC manifolds are simply connected by Lemma~\ref{lem:SimplyConnected}. 

As far as the strictness is concerned, Theorem~\ref{thm:coneendo2} allows us to restrict ourselves to the case $d=3$. A non-shellable constructible  $3$-ball with only $18$ facets is presented by Lutz in \cite{LutzEG}. Any homology $3$-sphere is Cohen-Macaulay, but it cannot be endo-collapsible, cf.~Corollary \ref{cor:poincare}.
Examples of non-constructible endo-collapsible $3$-spheres are given by Corollary~\ref{cor:EndoNoKnot2}.
For $3$-manifolds without boundary, the endo-collapsible notion coincides with the LC one: This follows directly from Corollary \ref{cor:poincare1}. 
In contrast, some manifolds with boundary are LC but not endo-collapsible. For example, the cubical complex given by a $3 \times 3 \times 3$ pile of cubes without the central one is LC because the removal of a cube makes it collapsible onto its boundary plus an edge. 
\end{proof}

\section*{Acknowledgements}
The author is very grateful to Eric Babson, Anders Bj\"{o}rner, Francesco Cavazzani, Alex Engstr\"{o}m, Rob Kirby,  and most of all G\"{u}nter Ziegler, for useful discussions, suggestions and references. Many thanks also to the MSRI, the Centro De Giorgi and the Berlin Mathematical School for hospitality and funding during this project. Thanks also to Maria Angelica Cueto for proofreading the Introduction.

\enlargethispage{2mm}
\begin{small}

\end{small}
\end{document}